\documentclass[11pt,reqno,a4paper]{amsart} 

\makeatletter
\def\@tocline#1#2#3#4#5#6#7{\relax
	\ifnum #1>\c@tocdepth 
	\else
	\par \addpenalty\@secpenalty\addvspace{#2}%
	\begingroup \hyphenpenalty\@M
	\@ifempty{#4}{%
		\@tempdima\csname r@tocindent\number#1\endcsname\relax
	}{%
		\@tempdima#4\relax
	}%
	\parindent\z@ \leftskip#3\relax \advance\leftskip\@tempdima\relax
	\rightskip\@pnumwidth plus4em \parfillskip-\@pnumwidth
	#5\leavevmode\hskip-\@tempdima
	\ifcase #1
	\or\or \hskip 1em \or \hskip 2em \else \hskip 3em \fi%
	#6\nobreak\relax
	\dotfill\hbox to\@pnumwidth{\@tocpagenum{#7}}\par
	\nobreak
	\endgroup
	\fi}
\makeatother

\usepackage[utf8]{inputenc}
\usepackage{amsmath}
\usepackage{amsthm}
\usepackage{amssymb}
\usepackage{euscript}
\usepackage{mathrsfs}
\usepackage{wasysym}
\usepackage[all]{xy}
\usepackage{graphicx}
\usepackage{color}
\usepackage{multirow}
\usepackage{pifont}
\usepackage[table]{xcolor}
\usepackage{tikz}
\usepackage{enumitem}
\usepackage{xcolor}
\usepackage{verbatim}
\usepackage{tikz-cd}
\usepackage[hidelinks]{hyperref}

\usepackage[babel]{csquotes}        
\usepackage[maxnames=99,style=alphabetic,backend=bibtex]{biblatex}
\AtBeginBibliography{\scriptsize}
\newtheorem{theorem}{Theorem}[section]
\newtheorem{corollary}[theorem]{Corollary}
\newtheorem{lemma}[theorem]{Lemma}
\newtheorem{proposition}[theorem]{Proposition}

\newtheorem{notation}[theorem]{Notation}
\newtheorem{definition}[theorem]{Definition}

\theoremstyle{definition}

\newtheorem{example}[theorem]{Example}

\newtheorem{remark}[theorem]{Remark}

 \newcommand{\N}{{\mathbb N}}
 \newcommand{\R}{{\mathbb R}}
\newcommand{\Q}{{\mathbb Q}}

 \newcommand{\HH}{\mathbb{H}}

\newcommand{\hh}{\mathbb{H}}
\newcommand{\oo}{\mathbb{O}}

\begin{document}
\title[]{On the first-order theories of quaternions\\ and octonions}

\author{Enrico Savi}
\address{Université Côte d'Azur - Laboratoire J. A. Dieudonné, Parc Valrose, 28 Avenue Valrose 06108 Nice (FRANCE)}
\email{enrico.savi@univ-cotedazur.fr}
\thanks{The author is supported by GNSAGA of INDAM, by the ANR NewMIRAGE (ANR-23-CE40-0002) and has been supported by ANR and IdEX of Universit\'e C\^ote d'Azur (IDEX UCAJedi ANR‑15‑IDEX‑01) during the development of this work.}


\begin{abstract}
	Let $L$ be the language of rings. We provide an axiomatization of the $L$-theories of quaternions and octonions and characterize their models: they coincide, up to isomorphism, with quaternion and octonion algebras over a real closed field, respectively. We bi-interpret these theories in terms of real closed fields and we prove they are complete, model complete and they do not have quantifier elimination. Then, we focus on the class of ordered polynomials. Over $\mathbb{H}$ and $\mathbb{O}$ these polynomials are of special interest in hypercomplex analysis since they are slice regular. We deduce some fundamental properties of their zero loci from model completeness and we introduce the notions of algebraic sets and Zariski topology. Finally, we prove the failure of quantifier elimination for the fragment of ordered formulas and we completely characterize the family of algebraic sets.
\end{abstract}

\maketitle 
\tableofcontents


\section*{Introduction}

\subsection*{Interactions between model theory \& geometry: a brief overview} In real and complex geometry, both algebraic and analytic, model theory constitutes a fundamental point of view on the subject. In the algebraic setting, the quantifier elimination result established by Chevalley in the complex case \cite{Che43} and by Tarski \cite{Tar51} in the real case constituted a milestone for a further development of algebraic geometry in these fields. In particular, modern real algebraic and semialgebraic geometry essentially arose after Tarski's result, whose consequences are also very clear in modern algorithmic algebraic geometry. For a very complete treatment on these topics in the real algebraic case we refer to \cite{BCR98,BPR06}.

In the real analytic setting an example by Osgood \cite{Os19} shows the existence of analytic maps whose image is not semianalytic, which implies the nonexistence of a similar result as Tarski's elimination in the real analytic setting. More in detail, a class of sets stable by boolean operations and by projection needs extra structure to be defined and is called subanalytic. After its definition by Gabrielov, the class of subanalytic sets and functions has been deeply studied from the point of view of geometry, see \cite{Hir73,BM88}. Since the 1980s, model theorists have focused on ordered structures, in particular on those sharing many finiteness properties with the ordered field of the reals which are called o-minimal structures. After their definition by Pillay and Steinhorn \cite{PS86}, the interactions of these structures with real analytic geometry has been deeply exploited by van den Dries, Miller and others, see as main references \cite{vdDM96,vdD98}. In particular, semialgebraic and subanalytic sets are special cases of o-minimal structures. We recall that new applications have recently appeared in complex analytic geometry as well, more precisely in Hodge theory, see the brand new papers \cite{BKT20,BBT23,BKU24}.

\subsection*{A model theoretical approach to hypercomplex analysis} The aim of this paper is to introduce and develop basic model theoretical properties of the division rings of quaternions and octonions with applications to hypercomplex analysis and geometry. This quite recent subject studies the analog of complex analysis in the more general setting of real alternative $*$-algebras of finite dimension. After many attempts during the 20th century, for instance by  Fueter, in 2007 Gentili and Struppa \cite{GS07} defined a notion of regular function $f:\mathbb{H}\to\mathbb{H}$, they called slice regular function, that generalizes to quaternions the classical concept of complex holomorphic function including, for the first time, quaternionic polynomials with coefficients on one side. This new notion of quaternionic regularity has generated a great deal of interest and has led both to a deep development of the theory in the quaternionic case, an almost complete reference is \cite{GSS13}, and to further extensions over very general real alternative $*$-algebras of finite dimension \cite{GP11}, including octonions and Clifford algebras. Then, Ghiloni and Perotti  \cite{GP22} developed the theory of slice regular functions in several variables. We outline two remarkable facts: slice regular functions are real analytic with respect to real coordinates and polynomials with ordered variables and coefficients on one side are slice regular. Consequently, very recent topics to investigate are the algebraic and geometric properties of algebraic subsets of $\hh^n$ and $\oo^n$ defined by slice regular polynomial equations. In particular, first results in this direction concern the description of the zero locus of a slice regular polynomial, see \cite{Lam01,GP22}, and the approach coming from commutative algebra, such as the quaternionic versions of Hilbert's Nullstellensatz in \cite{AP24,GSV24,GSV25}. The latter subjects strongly motivate the study of elimination results as Chevalley's and Tarski's theorems in the quaternionic and octonionic settings.

Let $L:=\{+,-,\cdot,0,1\}$ denote the language of rings. In this paper we provide an axiomatization of the $L$-theories of quaternions and octonions, ACQ and ACO for short, respectively, in Definitions \ref{def:q-structure}\,\&\,\ref{def:o-structure}. In Theorem \ref{thm:model-complete} and Corollary \ref{cor:complete} we prove  that these theories are complete and model complete but they do not admit quantifier elimination, see Example \ref{ex:qe}. We point out that, to the best of the author’s knowledge, such a model theory approach on the subject is not appearing in the literature, even though some of the main algebraic results used in the first two sections of the present paper go back to \cite{Niv41,Ghi12}. The main point of our approach follows from the possibility of expressing real coordinates in a first-order way by bi-interpreting each model of ACQ and ACO with respect to its center $R$, which turns out to be a real closed field. Thus, tools from real algebraic and semialgebraic geometry are available in both ACQ and ACO. Then, we study fundamental properties of what we call ordered polynomials with coefficients in a fixed model of ACQ and ACO. We point out that ordered polynomials over $\hh$ and $\oo$ coincide with the already mentioned slice regular polynomials. Here we prove that some algebraic and geometric properties of the zero loci of ordered polynomials extend from $\hh$ and $\oo$ to all models of their theories, see the Fundamental theorem of Algebra - Theorem \ref{thm:FTA} and Proposition \ref{prop:real-dim} below.

\subsection*{New perspectives and open problems} This paper opens to new perspectives from the point of view of model theory, algebraic geometry and hypercomplex analysis.

\vspace{.5em}

\noindent\textsc{Quantifier elimination: How to get it?} Very interesting questions to investigate are which fragments of ACQ and ACO do admit quantifier elimination and which extensions $L'$ of the language of rings $L$ leads the $L'$-theories of ACQ and ACO to admit quantifier elimination. First attempts in understanding these problems are already contained in the present paper. We prove that the fragment of ordered $L$-formulas does not admit quantifier elimination neither, in fact every $L$-formula is elementarily equivalent modulo ACQ and ACO to a formula contained in the fragment under consideration, see Theorem \ref{thm:ord-form}. In addition, in Remark \ref{rem:conj} we observe that the trace and the conjugation are $L$-definable functions with parameters in ACQ and ACO. Hence, a good candidate as a language $L'$ extending $L$ such that the $L'$-theories ACQ and ACO admit quantifier elimination is $L':=L\cup\{\textnormal{tr},^-\}$. Indeed, in the recent paper \cite{KT25} Klep and Tressl prove that the theory of square matrices $\mathcal{M}_n(R)$ of order $n$ over a real closed field $R$ admits quantifier elimination in the language $L'$ with parameters, where the interpretations of $\textnormal{tr}$ and $^-$ are the usual trace and transposition of matrices. Observe that, after embedding $\hh_R$ in the space of $(4\times 4)$-matrices $\mathcal{M}_4(R)$ in the standard way, we see that the trace and conjugate functions $\textnormal{tr}_{\hh}$ and $^-_\hh$ actually coincide (up to a constant) with transposition and trace of a matrix. Hence, this result strongly encourages the study of quantifier elimination for ACQ and ACO in the language $L':=L\cup\{^-,\textnormal{tr}\}$ as well.

\vspace{.5em}

\noindent\textsc{The geometry of algebraic sets.} In the present paper we introduce the notion of algebraic set over each model of ACQ and ACO as a finite union of sets defined by finite systems of ordered polynomial equations, see Definition \ref{def:algset}. We show that these sets constitute the closed sets of a topology, which we call the Zariski topology since it is the coarsest one whose family of closed sets contains the zero loci of ordered polynomials. Then, we characterize algebraic sets over each model of ACQ and ACO in terms of real algebraic sets, see Theorem \ref{thm:iso-real} below. However, for a fixed $m\in\N$, $H\models \textnormal{ACQ}$ and $O\models \textnormal{ACO}$, the family of algebraic subsets of $H^m$ and $O^m$ is, in general, strictly contained in the family of algebraic subsets of $R^{4m}$ and $R^{8m}$, as the case $m=1$ shows. This encourages the study of proper quaternionic and octonionic techniques to describe the geometry of algebraic sets for a fixed $m\in\N$. We expect that the recent results by Gori, Sarfatti, and Vlacci \cite{GSV24,GSV25} on the algebraic properties of quaternionic ordered polynomial rings and their ideals, in particular the Nullstellensatz-type theorems, will be useful for the further development of the subject. 

\vspace{.5em}

\noindent\textsc{Tame hypercomplex analysis.} Definable slice regular functions are expected to be a class of tame functions, both over $\hh$ and $\oo$, whose properties will extend to similar classes of definable functions over any model of ACQ and ACO, as suggested by model completeness for these $L$-theories and our results on ordered polynomials. In addition, since slice regular functions are in particular real analytic functions with respect to real coordinates, the class of definable slice regular functions actually coincides with the class of slice regular functions which are Nash with respect to real coordinates, see \cite[Definition\,2.9.3]{BCR98} and \cite{Shi87}. Recently, Bisi and Carbone \cite{BC25} independently introduced and studied this latter class of functions in the case of quaternionic and octonionic algebras over $\R$, establishing several equivalent definitions and proving many interesting finiteness properties. However, as already mentioned, sometimes a larger class of definable functions is needed to better encode more general properties of real analytic functions. More precisely, the analyticity of slice regular functions and the definability of real coordinates suggest the interest of exploiting methods from o-minimality as well into the subject. For this purpose, the results of Peterzil and Starchenko \cite{PS01,PS03} on definable complex analysis in a non-standard setting using o-minimality are fundamental starting points to investigate the quaternionic and octonionic cases for slice regular functions.

\subsection*{Structure of the paper} Let $L:=\{+,-,\cdot,0,1\}$ denote the language of rings.

In Section \ref{sec:1} we define the $L$-theories of algebraically closed quaternion and octonion algebras, ACQ and ACO for short, see Definitions \ref{def:q-structure}\&\ref{def:o-structure}. Then, we combine results by Niven and Ghiloni to completely characterize all possible models of these theories: up to isomorphism they coincide with quaternion and octonion algebras over a real closed field, see Theorem \ref{thm:niv-ghi}. Then, we provide a relative version of mentioned characterization with respect to substructures.

In Section \ref{sec:2} we prove the main results from the point of view of model theory. We prove that ACQ and ACO are bi-interpretable with RCF, the theory of real closed fields, see Proposition \ref{prop:bi-int}. As a consequence, we completely characterize all definable sets with parameters in a given model of ACQ and ACO, respectively, as semialgebraic sets in terms of the above bi-interpretation. Then, we prove that the $L$-theories ACQ and ACO are model complete, complete but they do not admit quantifier elimination, see Theorem \ref{thm:model-complete}, Corollary \ref{cor:complete} and Example \ref{ex:qe}.

In Section \ref{sec:3} we extend to every model of ACQ and ACO some properties of ordered polynomials already known in the literature over $\hh$ and $\oo$, respectively, such as the Fundamental Theorem of Algebra - Theorem \ref{thm:FTA} and a sharp bound for the real dimension of their zero loci in Proposition \ref{prop:real-dim} for polynomials in several variables. We introduce the notion of algebraic set and of the Zariski topology. Finally, we deduce that the fragment of ordered $L$-formulas does not admit quantifier elimination neither and, as a consequence of the proof, we give a complete characterization of algebraic sets both over ACQ and ACO in terms of real algebraic geometry.

\vspace{1.5em}

\section{Axiomatizations of ACQ $\&$ ACO}\label{sec:1}

Let $L:=\{+,-,\cdot,0,1\}$ denote the language of rings. The aim of this section is to provide two classes of first-order $L$-structures one containing $\hh$, the algebra of quaternions, and the other containing $\oo$, the algebra of octonions. At the end of the day, these classes of structures will exactly coincide with the class of models of the $L$-theory of quaternions $\textnormal{Th}_L(\HH)$ and of the $L$-theory of octonions $\textnormal{Th}_L(\oo)$, respectively, see Section \ref{subsec:2.1}.

\begin{notation}
	In what follows the term ring refers to a structure satisfying the usual ring axioms a part from associativity. Thus, in what follows a ring is non-necessarily associative.
\end{notation}

Let $D$ be a ring with unity $1\neq 0$. We say that $D$ is \emph{alternative} if for every $a,b\in D$ the following holds: $a(ab)=(aa)b$ and $(ab)b=a(bb)$. Hence, if $D$ is alternative no parenthesis are needed in the expressions $a^nb$ and $ab^n$. We say that $D$ is a \emph{division ring} if for every $a,b\in D$, with $a\neq 0$, there are unique $c,d\in D$ such that $ac=b$ and $da=b$. In particular, if $D$ is a division ring, then every non-null element $a\in D$ admits unique two-sided inverses, that is, there are unique $c,d\in D$ such that $ac=1$ and $da=1$. Denote by
\begin{equation*}
	R:=\{a\in D\,|\,\forall b\forall c((ab)c=a(bc)=(ca)b)\}
\end{equation*}
the \emph{center} of $D$, that is, the subring of $D$ constituting of all elements both associating and commuting with any other element of $D$. Observe that, when $D$ is associative, then $R$ can be defined as follows:
\[
R=\{a\in D\,|\,\forall b(ab=ba)\}
\]

\begin{remark}\label{rem:R-field}
	If $D$ is a division ring, then its center $R$ is a field. Indeed, by definition $R$ is a commutative and associative ring. We are left to prove that for every $a\in R$, with $a\neq 0$, the two-sided inverses, namely the unique $c,d\in D$ such that $ac=1$ and $da=1$, actually coincide and are contained in $R$. By associativity and commutativity of $a\in R$ with respect to any element of $D$ we obtain that $d=d(ac)=(da)c=c$, thus the two-sided inverses $c,d\in D$ of $a$ actually coincide. In addition, for every $b\in D$ we have $cb=(cb)(ac)=((cb)a)c=(a(cb))c=((ac)b)c=bc$ and for every $e,f\in D$ we have
	$(ce)f=c((e(ac))f)=c(((ea)c)f)=c(((ae)c)f)=(ca)((ec)f)=((ec)f)=(ce)f$, as desired.
\end{remark}

If $D$ is a division ring, by Remark \ref{rem:R-field}, $D$ is in particular a vector space over its center $R$. We say that $D$ is \emph{centrally finite} if its dimension as a $R$-vector space is finite.

\begin{example}
	Let $F$ be a field of characteristic $\neq 2$. Denote by $\hh_F$ and $\oo_F$ the rings of quaternions and octonions over $F$, respectively. More precisely, we define $\hh_F$ as the $F$-vector space of dimension $4$ with basis $\{1,i,j,k\}$ satisfying the multiplication rules:
	\begin{equation}\label{eq:q}
	i^2=j^2=k^2=-1\quad\text{and}\quad ij=k.
	\end{equation}
	We define $\oo_F$ as the $F$-vector space of dimension $8$ with basis $\{1,e_1,\dots,e_7\}$ satisfying the multiplication roles:
	\begin{equation}\label{eq:o}
	e_i^2=-1\quad\text{for}\,i=1,\dots,7\quad \text{and} \quad e_ie_j=\epsilon_{ijk} e_k\quad\text{for}\,i\neq j,
	\end{equation}
	where $\epsilon_{ijk}$ denotes a completely antisymmetric tensor with value $1$ when $ijk = 123, 145, 176, 246,$ $257, 347, 365$. That is, $\oo_F\simeq \hh_F\oplus e_4 \hh_F$, where a basis of $\hh_F$ as a $F$-vector space satisfying (\ref{eq:q}) is chosen to be $\{1,e_1,e_2,e_3\}$. We refer to \cite{EHHKMNPR91,Lam01,S66} for additional information on these algebras. Observe that, if $F$ is an ordered field (in particular $F$ has characteristic 0), then $\hh_F$ and $\oo_F$ are centrally finite alternative division rings.
\end{example}

Let $D$ be an alternative division ring and $R$ its center. Observe that in general the set of polynomials with coefficients in $D$ is not uniquely determined. One may define $D[\mathtt{q}]$ to be the set of those polynomials having coefficients in $D$ on the left of the variable $\mathtt{q}$ (or on the right, respectively) but then the product which induces to $D[\mathtt{q}]$ a structure of ring is not the pointwise one, so the classical evaluation map $e_a:D[\mathtt{q}]\to D$ at $a\in D$ defined as $e_a\left(\sum_h^m a_h\mathtt{q}^h\right):=\sum_h^m a_h a^h$ is not a ring homomorphism if $a\in D\setminus R$. On the other hand, one may define $D[\mathtt{q}]$ as the set of $L$-terms with coefficients in $D$ in the variable $\mathtt{q}$. In this case the pointwise sum and product endow $D[\mathtt{q}]$ with a ring structure and these operations behave well with the evaluation map $e_a:D[\mathtt{q}]\to D$ at $a\in D$ which is actually a ring homomorphism but in that case polynomials are way much more complicated. However, since elements of $R$ actually commute with every element of $D$, such a choice is unique when we define the ring of polynomials $R[\mathtt{q}]$ in the variable $\mathtt{q}$ with coefficients in the center $R$ of $D$.

We say that $D$ is \emph{algebraically closed} if every nonconstant polynomial $p(\mathtt{q})\in R[\mathtt{q}]$ has a root in $D$. For instance, both $\hh$ and $\oo$ are algebraically closed alternative division rings, see \cite{Niv41} and its further topological generalization in \cite{EN44} for quaternions and \cite{Ghi12} for octonions. In addition, we refer to \cite{Lam01,Ghi12} for a complete description of several algebraically closure conditions, which are actually all equivalent for a centrally finite alternative division ring $D$.

\begin{remark}
	Recall that $R:=\{a\in D\,|\,\forall b\forall c((ab)c=a(bc)=(ca)b)\}$, so it is an $L$-definable set. This means that the algebraically closure condition for an alternative division ring $D$ defined above can be expressed by a countable set of $L$-sentences. Indeed, next $L$-sentence means that each polynomial $p(\mathtt{q}):=\sum_{h=1}^{d}a_h\mathtt{q}^h\in R[\mathtt{q}]$ of fixed degree $d\geq 1$ has a root in $D$:
	\[
	\forall a_0\dots \forall a_d \left(\Big(a_0,\dots,a_d\in R\wedge (a_d\neq 0)\Big)\rightarrow\Big(\exists b \Big(\sum_{h=1}^{d}a_h b^h=0\Big)\Big)\right).
	\]
\end{remark}

Now we are in position to define the classes of $L$-structures containing $\hh$ and $\oo$, respectively, we are interested in.

\begin{definition}[Algebraically closed quaternion algebra - ACQ]\label{def:q-structure}
	Let $H$ be an $L$-structure. Denote by $R:=\{q\in H\,|\,\forall a\forall b ((qa)b=q(ab)=(bq)a)\}$ the center of $H$. We say that $H$ is an \emph{algebraically closed quaternion algebra} if the following axioms are satisfied:
	\begin{itemize}
		\item[\emph{(H$1$)}]\label{h1} $H$ is an alternative division ring with unity such that $1\neq 0$.
		\item[\emph{(H$2$)}]\label{h2} $H$ is centrally finite of dimension $4$, that is, there are $i,j,k\in H$ so that $\{1,i,j,k\}$ is a basis of $H$ as a vector space over its center $R$. As a first-order $L$-formula it is described as:
		\begin{align*}
			\quad\exists i\exists j\exists k &\big( \neg(\exists\lambda_1\exists\lambda_i\exists\lambda_j\exists\lambda_k((\lambda_1,\lambda_i,\lambda_j,\lambda_k\in R)\wedge \neg(\lambda_1=\lambda_i=\lambda_j=\lambda_k=0)\\
			&\wedge (0=\lambda_1+\lambda_i\cdot i+\lambda_j\cdot j+\lambda_k k))\wedge\\
			&\forall q\exists\lambda_1\exists\lambda_i\exists\lambda_j\exists\lambda_k((\lambda_1,\lambda_i,\lambda_j,\lambda_k\in R)\wedge(q=\lambda_1+\lambda_i\cdot i+\lambda_j\cdot j+\lambda_k\cdot k))\big).
		\end{align*}
		\item[\emph{(H$3$)}]\label{h_3} $H$ is algebraically closed.
	\end{itemize}
	We denote by $\textnormal{ACQ}$ the \emph{$L$-theory of algebraically closed quaternion algebras}.
\end{definition}

\begin{definition}[Algebraically closed octonion algebra - ACO]\label{def:o-structure}
	Let $O$ be an $L$-structure. Denote by $R:=\{o\in O\,|\,\forall a\forall b ((oa)b=o(ab)=(bo)a)\}$ the center of $O$. We say that $O$ is an \emph{algebraically closed octonion algebra} if the following axioms are satisfied:
	\begin{itemize}
		\item[\emph{(O$1$)}] $O$ is an alternative division ring with unity such that $1\neq 0$.
		\item[\emph{(O$2$)}] $O$ is centrally finite of dimension $8$, that is, there are $e_1,e_2,e_3,e_4,e_5,e_6,$ $e_7\in O$ so that $\{1,e_1,e_2,e_3,e_4,e_5,e_6,e_7\}$ is a basis of $O$ as a vector space over its center $R$. As in the case of axiom $(\textnormal{H}2)$, previous property can be expressed as a first-order $L$-formula.
		\item[\emph{(O$3$)}] $O$ is algebraically closed.
	\end{itemize}
	We denote by $\textnormal{ACO}$ the \emph{$L$-theory of algebraically closed octonion algebras}.
\end{definition}

\begin{remark}
	Observe that, in general, to define a notion of vector space over a field a richer language is needed. Indeed, in general, one needs to define a functional symbol for each scalar multiplication with respect to an element of the ground field. However, in our case the language of rings $L:=\{+,-,\cdot,0,1\}$ is sufficient since the field $R$ such that $H$ and $O$ are $R$-vector spaces is an $L$-definable subset both in $H$ and $O$.
\end{remark}

Denote by $\textnormal{RCF}$ the first-order $L$-theory of real closed fields. Let us recall the precise $L$-axioms for that theory.

\begin{definition}[Real closed field - RCF]
	Let $R$ be an $L$-structure. We say that $R$ is a \emph{real closed field} if the following axioms are satisfied:
	\begin{itemize}
		\item[\emph{(R$1$)}] $R$ is a formally real field, that is, $R$ is a field and $-1$ cannot be expressed as a finite sum of squares.
		\item[\emph{(R$2$)}] Each polynomial of odd degree has a root in $R$.
		\item[\emph{(R$3$)}] $\forall x\exists y (x=y^2 \vee x=-y^2)$.
	\end{itemize}
\end{definition}

Let $S_1$ and $S_2$ be two $L$-structures. With abuse of notation, denote that $S_1$ is an $L$-\emph{substructure} of $S_2$ as $S_1\subseteq S_2$. Similarly, we denote $S_1$ is an \emph{elementary $L$-substructure} of $S_2$ as $S_1\preceq S_2$.

Next result is a well-known lemma concerning model completeness and extensions of languages: if an $L$-theory $T$ is model complete in an extension $L'$ of $L$, then $T$ is model complete. We give a proof in the case of RCF for sake of completeness.

\begin{lemma}\label{lem:RCF}
	The $L$-theory \textnormal{RCF} is model complete.
\end{lemma}

\begin{proof}
	Let $R_1,R_2\models \textnormal{RCF}$ such that $R_1\subseteq R_2$. Let us consider the extended language $L_<:=L\cup\{<\}$. Consider the $L_<$-structures $\langle R_1,<_1\rangle$ and $\langle R_2,<_2 \rangle$  in which $<_i$ is defined as follows: $x<_iy$ iff $\exists z(z\neq 0 \wedge x+z^2=y)$, for $i=1,2$. Observe that, since $R_1$ is an $L$-substructure of $R_2$, then $<_2|_{R_1}=<_1$. Denote by $\textnormal{RCF}_<$ the $L_<$-theory of real closed fields. Thus, $\langle R_i,<_i\rangle\models\textnormal{RCF}_<$ for $i=1,2$ and $\langle R_1,<_1\rangle\subseteq \langle R_2,<_2\rangle$. By quantifier elimination of $\textnormal{RCF}_<$, see  \cite{Tar51}, we have in particular $\langle R_1,<_1\rangle\preceq\langle R_2,<_2\rangle$. Hence: 
	\[
	R_1=\langle R_1,<_1\rangle|_L\preceq \langle R_2,<_2\rangle|_L=R_2.
	\]
\end{proof}

As a consequence of \cite[Theorem\,1]{Niv41} and \cite[Theorem 16.15]{Lam01}, originally proved by Baer, for the quaternion case and \cite[Theorem\,1.2]{Ghi12} for the octonion case, we have the following characterization of algebraically closed quaternion and octonion algebras:

\begin{theorem}[Niven-Baer, Ghiloni]\label{thm:niv-ghi}
	Let $H$ and $O$ be alternative division rings and denote by $R$ their centers. Then:
	\begin{enumerate}[label=\emph{(\roman*)}, ref=(\roman*)]
		\item $H\models \textnormal{ACQ}$ if and only if $R\models\textnormal{RCF}$ and $H$ is isomorphic to the quaternion algebra $\hh_R$ over $R$.
		\item $O\models \textnormal{ACO}$ if and only if $R\models\textnormal{RCF}$ and $O$ is isomorphic to the octonion algebra $\oo_R$ over $R$.
	\end{enumerate}
\end{theorem}

\begin{remark}
	Observe that as a consequence of Theorem \ref{thm:niv-ghi} and the extension of Frobenius Theorem over every real closed field, see \cite[Theorem 13.12, p. 208-9]{Lam01}, the axiom (H2) of Definition \ref{def:q-structure} can be replaced by the following:
	\begin{itemize}
		\item[(H$2$')]\label{h2'} $R\models \textnormal{RCF}$ and $H$ is noncommutative.
	\end{itemize}
\end{remark}

\begin{remark}
	By Theorem \ref{thm:niv-ghi}, we observe that actually in the axioms Definition \ref{def:q-structure}(H2) and Definition \ref{def:o-structure}(O2) we may safely require in addition that the elements $i,k,j\in H$ of the basis of $H$ over $R$ and $e_1,\dots,e_7$ of the basis of $O$ over $R$ satisfy the relations in (\ref{eq:q}) and (\ref{eq:o}), respectively. We could also have required $H$ to be associative in Definition \ref{def:q-structure}(H1) and $O$ to be non-associative in Definition \ref{def:o-structure}(O1) again by Theorem \ref{thm:niv-ghi}. Hence, in particular, if $H\models \textnormal{ACQ}$ then its center $R$ can be described by the simplified formula $\forall a (a\mathtt{q}=\mathtt{q}a)$.
\end{remark}

Let us specify the previous theorem in terms of extensions of $L$-structures.

\begin{lemma}\label{lem:real-sub}
	Let $H\models \textnormal{ACQ}$ and $O\models \textnormal{ACO}$. Denote by $R$ the center of $H$ and $O$. Then, for every $R_1\models \textnormal{RCF}$ such that $R_1\preceq R$, $\hh_{R_1}$ can be embedded as an $L$-substructure of $H$ and $\oo_{R_1}$ can be embedded as an $L$-substructure of $O$ in such a way that the embeddings $\varphi_{\hh_{R_1}}:\hh_{R_1}\to H$ and $\varphi_{\oo_{R_1}}:\oo_{R_1}\to O$ can be extended to isomorphisms $\varphi_{\hh_{R}}:\hh_{R}\to H$ and $\varphi_{\oo_{R}}:\oo_{R}\to O$.
\end{lemma}

\begin{proof}
	Let $\hh_{R_1}\subseteq\hh_R=\hh_{R_1}\otimes_{R_1} R$ and $\oo_{R_1}\subseteq\oo_R=\oo_{R_1}\otimes_{R_1} R$. By Theorem \ref{thm:niv-ghi}, there are isomorphisms $\psi_H:H\to\hh_R$ and $\psi_O:O\to\oo_R$, so just consider the embeddings $\varphi_{\hh_{R_1}}:=\psi_H^{-1}|_{\hh_{R_1}}$ and $\varphi_{\oo_{R_1}}:=\psi_O^{-1}|_{\oo_{R_1}}$.
\end{proof}

\begin{remark}\label{rem:Q-bar}
	Denote by $\overline{\Q}^r:=\overline{\Q}\cap \R$ the real closure of $\Q$. As a consequence of Lemma \ref{lem:real-sub}, we have that $\hh_{\overline{\Q}^r}$ can be embedded as an $L$-substructure of $H$ and $\oo_{\overline{\Q}^r}$ can be embedded as an $L$-substructure of $O$ in such a way that the image of $\overline{\Q}^r$ is contained in the center $R$ of $H$ or $O$, respectively, and the embedding extends to an isomorphism as in Lemma \ref{lem:real-sub}, for every $H\models\textnormal{ACQ}$ and $O\models\textnormal{ACO}$.
\end{remark}

\begin{lemma}\label{lem:extension}
	Let $H_1,H_2\models \textnormal{ACQ}$ and $O_1,O_2\models \textnormal{ACO}$. The following hold:
	\begin{enumerate}[label=\emph{(\roman*)}, ref=(\roman*)]
		\item\label{lem:extension_1} Denote by $R_i$ the center of $H_i$ for $i=1,2$. If $H_1\subseteq H_2$, then $R_1\preceq R_2$ and each isomorphism $\varphi_1:H_1\to\hh_{R_1}$ extends to an isomorphism $\varphi_2:H_2\to\hh_{R_2}$. 
		\item\label{lem:extension_2}  Denote by $R_i$ the center of $O_i$ for $i=1,2$. If $O_1\subseteq O_2$, then $R_1\preceq R_2$ and each isomorphism $\varphi_1:O_1\to\oo_{R_1}$ extends to an isomorphism $\varphi_2:O_2\to\oo_{R_2}$. 
	\end{enumerate}
	In particular, in both \emph{\ref{lem:extension_1}\&\ref{lem:extension_2}}, for every $\phi\in\textnormal{Gal}(R_2:R_1)$ there exists a unique extension $\varphi_2$ of $\varphi_1$ such that $\varphi_2|_{R_2}=\phi$.
\end{lemma}

\begin{proof}
	Let us first prove \ref{lem:extension_1}. Denote by $1_{H}$ the identity element of $H_1$ and $H_2$ and by $i_H:=\varphi_1^{-1}(i), j_H:=\varphi_1^{-1}(j), k_H:=\varphi_1^{-1}(k)\in H_1\subset H_2$. Observe that $i_H,j_H,k_H\in H_1$ satisfy Definition \ref{def:q-structure}$(\textnormal{H}2)$ for $H_1$ and the relations (\ref{eq:q}). We will prove that $R_1\preceq R_2$ and $i_H,j_H,k_H\in H_1$ satisfy Definition \ref{def:q-structure}$(\textnormal{H}2)$ for $H_2$ as well.
	
	Since $H_1\subseteq H_2$, we have $\Q\subset R_1\cap R_2$. Consider the real closure $\overline{\Q}^r:=\overline{\Q}\cap\R$ of $\Q$. Recall that, by Lemma \ref{lem:RCF}, we have $\overline{\Q}^r\preceq R_1$ and $\overline{\Q}^r\preceq R_2$. By Lemma \ref{lem:real-sub} and Remark \ref{rem:Q-bar}, there exists $H\subset H_1$ such that $H$ is isomorphic to $\hh_{\overline{\Q}^r}$, $i_H,j_H,k_H\in H$ and they satisfy Definition \ref{def:q-structure}$(\textnormal{H}2)$ for $H$ and the relations (\ref{eq:q}). Let us consider $H\subseteq H_2$. Let $i_H\in H$. Observe that, since $i_H\notin {\overline{\Q}^r}$ but it is algebraic over $\Q$ and ${\overline{\Q}^r}\preceq R_2$, we have that $i_H\notin R_2$ as well. More precisely, $i_H:=\varphi_1^{-1}(i)$ is a square root of unity since $i$ is so and $\varphi_1|_{R_1[i]}$ is an isomorphism. Hence, ${\overline{\Q}^r}[i_H]$ and $R_2[i_H]$ are algebraically closed fields of characteristic $0$ such that ${\overline{\Q}^r}[i_H]\preceq R_2[i_H]$ as models of ACF, the $L$-theory of algebraically closed fields, which has quantifier elimination as well (see \cite{Che43}). In addition, $H={\overline{\Q}^r}[i_H]\oplus j_H {\overline{\Q}^r}[i_H]$. Observe that $j_H, k_H:=i_H j_H\notin R_2[i_H]$, otherwise $R_2[i_H]$ would not be commutative, thus we have $H_2=R_2[i_H]\oplus j_H R_2[i_H]$ as well by Frobenius' Theorem. This proves that $R_1\preceq R_2$ and the map $\varphi_2:H_2\to \hh_{R_2}$ defined by $R_2$-linear combinations of $\varphi_2(1_H)=1$, $\varphi_2(i_H)=i$, $\varphi_2(j_H)=j$ and $\varphi_2(k_H)=k$ is an isomorphism extending $\varphi_1$. More in detail, for every $\phi\in\textnormal{Gal}(R_2:R_1)$, there is a unique isomorphism $\varphi_2:H_2\to \hh_{R_2}$ such that $\varphi_2|_{H_1}=\varphi_1$ and $\varphi_2|_{R_2}=\phi$ constructed as above.
	
	The strategy to prove \ref{lem:extension_2} is the same as in \ref{lem:extension_1}. Denote by $1_{O}$ the identity element of $O_1$ and $O_2$ and by $e_{\ell,O}:=\varphi_1^{-1}(e_{\ell})\in O_1\subset O_2$, for every $\ell=1,\dots,7$, where $\{1,e_1,\dots,e_7\}$ denotes the canonical basis of $\oo_{R_1}$. Since $O_1\subseteq O_2$, we have $\Q\subset R_1\cap R_2$. Then, $\overline{\Q}^r\preceq R_2$ as well. Observe that $\{e_{1,O},\dots,e_{7,O}\}\in O_1$ satisfy Definition \ref{def:o-structure}$(\textnormal{O}2)$ for $O_1$ and the relations (\ref{eq:o}). Let $O\subset O_1$ so that $O$ is isomorphic to $\oo_{\overline{\Q}^r}$, $e_{\ell,O}\in O$ for every $\ell=1,\dots,7$ and $\{e_{1,O},\dots,e_{7,O}\}$ satisfy Definition \ref{def:o-structure}$(\textnormal{O}2)$ for $O$. Recall that $\oo_{\overline{\Q}^r}=\hh_{\overline{\Q}^r}\oplus e_4 \hh_{\overline{\Q}^r}$, in which $i=e_{1}$, $j=e_{2}$ and $k=e_{3}$. So following the same argument of \ref{lem:extension_1}, we get a real algebra $H_2=R_2[e_{1,O}]\oplus e_{2,O} R_2[e_{1,O}]\subset O_2$ isomorphic to $\hh_{R_2}$. Observe that $e_{4,O}\notin H_2$, otherwise $H_2$ would not be associative. Then, we have $O_2=H_2\oplus e_{4,H}H_2$. This proves that $R_1\preceq R_2$ and the map $\varphi_2:O_2\to \oo_{R_2}$ defined by $R_2$-linear combinations of $\varphi_2(1_O)=1$, $\varphi_2(e_{\ell,O})=e_\ell$ for $\ell=1,\dots,7$ is an isomorphism extending $\varphi_1$. More in detail, for every $\phi\in\textnormal{Gal}(R_2:R_1)$, there is a unique isomorphism $\varphi_2:O_2\to \hh_{R_2}$ such that $\varphi_2|_{O_1}=\varphi_1$ and $\varphi_2|_{R_2}=\phi$ constructed as above.
\end{proof}

As a consequence of Theorem \ref{thm:niv-ghi}, up to isomorphism, we may suppose every model of ACQ is of the form $\hh_R$ and every model of ACO is of the form $\oo_R$, for some real closed field $R$. In addition, by Lemma \ref{lem:extension} we may suppose for every $H_1,H_2\models \textnormal{ACQ}$ and $O_1,O_2\models \textnormal{ACO}$ such that $H_1\subseteq H_2$ and $O_1\subseteq O_2$ there exist $R_1,R_2\models \textnormal{RCF}$ such that $R_1\preceq R_2$, $H_1=\hh_{R_1}\subseteq \hh_{R_2}=H_2$ and $O_1=\oo_{R_1}\subseteq \oo_{R_2}=O_2$. So, from now on we will only deal with quaternion and octonion algebras over real closed fields as models of ACQ and ACO, respectively, even in extension arguments.

\vspace{1.5em}
\section{Basic properties of ACQ $\&$ ACO}\label{sec:2}

This section is devoted to establish model theoretical properties of ACQ and ACO descending from the real vector structure over its center. As already mentioned, we will consider models $H\models \textnormal{ACQ}$ and $O\models\textnormal{ACO}$ up to isomorphism, that is, we will suppose $H=\hh_R$ and $O=\oo_R$ for some $R\models\textnormal{RCF}$ and whenever $H_1, H_2\models \textnormal{ACQ}$ and $O_1,O_2\models \textnormal{ACO}$ such that $H_1\subseteq H_2$ and $O_1\subseteq O_2$, then $H_i=\hh_{R_i}$ and $O_i=\oo_{R_i}$ for some $R_i\models\textnormal{RCF}$ for $i=1,2$ and $R_1\preceq R_2$.

\vspace{0.5em}

\subsection{Bi-interpretation \& characterization of definable sets}\label{subsec:2.1}  We start by establishing the mutual interpretations of centrally finite division algebras with respect to their centers. 

\begin{proposition}[Mutual interpretation]\label{prop:m-int}
	Let $A$ be a centrally finite division algebra with center $R_1$, then $A$ and $R_1$ are mutually-interpretable in the language of rings $L$ with parameters. 
	
	In addition, the mutual interpretation is uniform in the sense that it preserves extensions of models: the interpretations $\iota:R_1\to A$ and $\jmath:A\to R_1$ extend to interpretations $\iota':R_2\to A_{R_2}$ and $\jmath':A_{R_2}\to {R_2}$ of $R_2$, where $A_{R_2}:=A\otimes_{R_1} R_2$ for every field extension $R_2|R_1$.
\end{proposition}

\begin{proof}
	Observe that the field $R_1$ as an $L$-structure is isomorphic to the center of $A$ endowed with the restriction of the operations and the constant symbols of $L$ in $A$ to its center. This provides an interpretation $\iota:R_1\to A$ which extends to an interpretation $\iota_2:R_2\to A_{R_2}:=A\otimes_{R_1} R_2$, with $A\subseteq A_{R_2}$.
	
	Denote by $n$ the dimension of $A$ as a $R_1$-vector space and let $\{e_1=1,e_2,,\dots,e_{n}\}$ be a basis of $A$ over $R_1$. Let $a^{(i)}_{jk}\in R_1$ such that $e_j\cdot e_k=\sum_{i=1}^n a^{(i)}_{jk} e_i$ every $j,k\in\{1,\dots,n\}$ and $i=1,\dots,n$. Let us define the appropriate $L$-definable sum $+:R_1^n\times R_1^n\to R_1^n$, minus $-:R_1^n\to R_1^n$ and product  $*:R_1^n\times R_1^n\to R_1^n$ in $R_1^n$ with neutral elements $\underline{0}:=(0,0,\dots,0)$ and $\underline{1}:=(1,0,\dots,0)$ in such a way that $\langle R_1^n, +,-,*,\underline{0},\underline{1} \rangle$ is isomorphic to $\langle A,+,-,\cdot,0,1 \rangle$.  Define these operations as follows:
	\begin{align*}
		(a_1,a_2,\dots,a_n)+(b_1,b_2,\dots,b_n):=&(a_1+b_1,a_2+b_2,\dots,a_n+b_n),\\
		-(a_1,a_2,\dots,a_n):=&(-a_1,-a_2,\dots,-a_n),\\
		(a_1,a_2,\dots,a_n)*(b_1,b_2,\dots,b_n):=&\big((a_1,a_2,\dots,a_n)^T A_1(b_1,b_2,\dots,b_n),\\&\,\,\,(a_1,a_2,\dots,a_n)^T A_2(b_1,b_2,\dots,b_n),\\ &\quad\quad\quad\quad\quad\quad\quad\quad\vdots\\
		&\,\,\,(a_1,a_2,\dots,a_n)^T A_{n}(b_1,b_2,\dots,b_n)\big),
	\end{align*}
	where $A_i=(a^{(i)}_{jk})_{j,k=1,\dots,n}\in\mathcal{M}_{n}(R_1)$ for every $i\in\{1,\dots,n\}$. Then, define the isomorphism of $L$-structures $\jmath:A\to R_1^n$ as $\varphi(\sum_{i=1}^n a_i e_i):=(a_1,a_2,\dots,a_n)$. Observe that the previous operations are $L$-definable with parameters in $R_1^n$ and they trivially extend to operations over $R_2^n$. Hence the map $\jmath:A_{R_2}\to R_2^n$ as $\varphi(\sum_{i=1}^n a_i e_i):=(a_1,a_2,\dots,a_n)$ is still an isomorphism between $\langle A_{R_2}, +,-,\cdot,0,1 \rangle$ and $\langle R_2^n, +,-,*,\underline{0},\underline{1} \rangle$, as desired. 
\end{proof}

More precisely, in the case of algebraically closed quaternion and octonion algebras we obtain bi-interpretability as well.

\begin{proposition}[Bi-interpretation]\label{prop:bi-int}
	Let $R\models\textnormal{RCF}$. Then, $\hh_R$ and $R$ and $\oo_R$ and $R$ are bi-interpretable in the language of rings with parameters.
	
	In addition, the bi-interpretation is uniform in the sense that it preserves extension of models: let $R_1,R_2\models \textnormal{RCF}$ such that $R_1\preceq R_2$, then the interpretations $\iota:R_1\to\hh_{R_1}$ and $\jmath:\hh_{R_1}\to R$ extend to interpretations $\iota_2:R_2\to\hh_{R_2}$ and $\jmath:\hh_{R_2}\to {R_2}$.
\end{proposition}

\begin{proof}
	Let us show that the mutual interpretations $\iota:R_1\to \hh_{R_1}$ and $\jmath:\hh_{R_1}\to R_1^4$ of Proposition \ref{prop:m-int} provide two $L$-definable functions $\jmath\circ\iota$ and $\iota^4\circ\jmath$, possibly with parameters. Observe that the map $\jmath\circ\iota:R_1\to\R_1^4$ is $\jmath\circ\iota(x)=(x,0,0,0)$, thus it is $L$-definable. On the other hand, the function $\iota^4\circ\jmath:\hh_{R_1}\to \hh_{R_1}^4$ is given by $(\iota^4\circ\jmath)(a_0+a_1 i+a_2 j+a_3 k):=(a_0,a_1,a_2,a_3)$, hence the graph $\Gamma(\iota^4\circ\jmath)\subset\hh_{R_1}\times\hh_{R_1}^4$ of $\iota^4\circ\jmath$ is defined by the equations
	\begin{align}
		\label{eq:bi-q_1}4 \mathtt{q}_{1}&=\mathtt{q}-i\mathtt{q} i-j\mathtt{q} j-k\mathtt{q} k,\quad 4i \mathtt{q}_2=\mathtt{q}-i\mathtt{q} i+j\mathtt{q} j+k\mathtt{q} k,\\
		\label{eq:bi-q_2}4j \mathtt{q}_3&=\mathtt{q}+i\mathtt{q} i-j\mathtt{q} j+k\mathtt{q} k,\quad 4k \mathtt{q}_4=\mathtt{q}+i\mathtt{q} i+j\mathtt{q} j-k\mathtt{q} k.
	\end{align}
	Hence, the function $\iota^4\circ\jmath:\hh_{R_1}\to \hh_{R_1}^4$ is $L$-definable with parameters.
	
	Similarly, the mutual interpretations $\iota:R_1\to \oo_{R_1}$ and $\jmath:\oo_{R_1}\to R_1^8$ of Proposition \ref{prop:m-int} provide two $L$-definable functions $\jmath\circ\iota$ and $\iota^8\circ\jmath$, eventually with parameters. Observe that the map $\jmath\circ\iota:R_1\to\R_1^8$ is $\jmath\circ\iota(x)=(x,0,\dots,0)$, thus it is $L$-definable. On the other hand, the function $\iota^8\circ\jmath:\oo_{R_1}\to \oo_{R_1}^4$ is given by $(\iota^8\circ\jmath)(\sum_{i=0}^{7}a_i e_i):=(a_0,\dots,a_7)$, hence the graph $\Gamma(\iota^8\circ\jmath)\subset\oo_{R_1}\times\hh_{R_1}^8$ of $\iota^8\circ\jmath$ is defined by the equations
	\begin{align}
		\label{eq:bi-o}8 \mathtt{o}_{1}&=\mathtt{o}-\sum_{h=1}^7 e_h\mathtt{o} e_h,\quad 8 e_\ell \mathtt{o}_{\ell}=\mathtt{o}-e_\ell\mathtt{o} e_\ell+\sum_{h\neq \ell} e_h\mathtt{o} e_h \quad\text{for every }\ell\in\{1,\dots,7\}.
	\end{align}
	Hence, the function $\iota^8\circ\jmath:\oo_{R_1}\to \oo_{R_1}^8$ is $L$-definable with parameters.
\end{proof}

\begin{remark}\label{rmk:L'}
	Observe that the elements $i,j,k\in\hh_R$ and $\{e_h\}_{h=1,\dots,7}\in\oo_R$ are not $L$-definable. Let us focus on the quaternionic case, a similar argument works for octonions as well. Suppose $i\in\hh_R$ is an $L$-definable element, then restricting the $L$-formula defining $i$ to the complex plane $C=R\oplus Ri$ gives, as a consequence, that $i$ is $L$-definable in $C$ as well. However, this is impossible by unique factorization of polynomials over $\Q$ since $\mathtt{x}^2+1\in\Q[\mathtt{x}]$ is the minimal polynomial of $i$ in $\Q[\mathtt{x}]$. This proves that it was necessary to add at least the constants $i,j,k\in\hh_R$ and $\{e_h\}_{h=1,\dots,7}\in\oo_R$ to the language $L$ to obtain equations (\ref{eq:bi-q_1})-(\ref{eq:bi-o}). We point out that formulas (\ref{eq:bi-q_1})\&(\ref{eq:bi-q_2}) were already known in literature, see \cite{GS12} as a reference, whereas formulas (\ref{eq:bi-o}) seem to appear for the first time, even though they derive from a direct computation similar to the quaternionic case.
\end{remark}

As a consequence of the bi-interpretation of the theories ACQ, ACO and RCF we get the following description of definable sets with parameters.

\begin{corollary}\label{cor:bi-int}
	Let $A$ be a division ring of dimension $d$ over its center $R$ and denote by $\jmath:A\to R^d$ as in Proposition \ref{prop:m-int}. Then, for every $n\in \N$, $\jmath^n(X)\subset\R^{nd}$ is $L$-definable whenever $X\subset A^n$ is so and $\jmath^n(X)\subset\R^{nd}$ is quantifier-free $L$-definable whenever $X\subset A^n$ is so.
	
	In addition, if $A\models\textnormal{ACQ}$ or $A\models\textnormal{ACO}$, then $A$ is isomorphic to $\hh_{R}$ or $\oo_R$, respectively, for some $R\models \textnormal{RCF}$. Then:
	\begin{enumerate}[label=\emph{(\roman*)}, ref=(\roman*)]
		\item\label{cor:bi-q} A set $X\subset \hh_{R}^n$ is $L$-definable with parameters if and only if $\jmath^n(X)\subset R^{4n}$ is semialgebraic.
		\item\label{cor:bi-o} A set $X\subset \oo_{R}^n$ is $L$-definable with parameters if and only if $\jmath^n(X)\subset R^{8n}$ is semialgebraic.
	\end{enumerate}
\end{corollary}

\begin{proof}
	The first part of the statement follows directly from the fact that the operations of $\langle R^n,+,-,*,\underline{0},\underline{1} \rangle$ defined in Proposition \ref{prop:m-int} are quantifier-free $L$-definable. On the other hand, assume $A=\hh_R$ (up to isomorphism) and consider the bi-interpretations $\iota:R\to \hh_{R}$ and $\jmath:\hh_R\to R^4$ of Proposition \ref{prop:bi-int}. Let $X\subset\hh_R^n$ such that $\jmath^n(X)\subset R^{4n}$ is semialgebraic. Since $\iota$ is an interpretation, $\iota^{4n}:R^{4n}\to\hh_R^{n}$ satisfies $Y:=\iota^{4n}(\jmath^n(X))$ is definable. Hence, $X=(\iota^{4n}\circ\jmath^n)^{-1}(Y)$ is definable since $\iota^{4n}\circ\jmath^n:\hh_R^n\to\hh_R^{4n}$ is a definable isomorphism by Proposition \ref{prop:bi-int}. Assume $A=\oo_R$ (up to isomorphism) and consider the bi-interpretations $\iota:R\to \oo_{R}$ and $\jmath:\oo_R\to R^8$ of Proposition \ref{prop:bi-int}. Let $X\subset\oo_R^n$ such that $\jmath^n(X)\subset R^{8n}$ is semialgebraic. Denote by $Y:=\iota^{8}(\jmath^n(X))$. As before, $X=(\iota^{4n}\circ\jmath^n)^{-1}(Y)$ is definable since $\iota^{8n}\circ\jmath^n:\hh_R^n\to\hh_R^{4n}$ is a definable isomorphism by Proposition \ref{prop:bi-int} and $Y$ is definable since $\iota$ is an interpretation.
\end{proof}

\begin{remark}
	Observe that when interpreting $R\models\textnormal{RCF}$ into $\hh_R$ (or $\oo_R$) the property of a set of being quantifier-free $L$-definable in $R$ is in general not preserved in $\hh_R$ (or $\oo_R$). Indeed, the restriction of an $L$-formula on the center $R$ of $\hh_R$ (or $\oo_R$) involves the use of quantifiers in the language of rings $L:=\{+,-,\cdot,0,1\}$.
\end{remark}

\vspace{.5em}

\subsection{Model Completeness $\&$ Completeness}

Here we are in a position to prove the main properties of the $L$-theories ACQ and ACO contained in this paper. We stress that whenever $H_1,H_2\models\textnormal{ACQ}$ the center $R_i$ of $H_i$ satisfies $R_i\models\textnormal{RCF}$ for $i=1,2$ and, in Lemma \ref{lem:extension}, we have proved that $R_1\preceq R_2$ and there is an isomorphism $\varphi:H_2\to\hh_{R_2}$ such that $\varphi|_{H_1}:H_1\to\hh_{R_1}$ is an isomorphism. Similarly, whenever $O_1,O_2\models\textnormal{ACO}$ the center $R_i$ of $O_i$ satisfies $R_i\models\textnormal{RCF}$ for $i=1,2$, $R_1\preceq R_2$ and there is an isomorphism $\varphi:O_2\to\oo_{R_2}$ such that $\varphi|_{O_1}:O_1\to\oo_{R_1}$ is an isomorphism. Hence, we may safely restrict to quaternion and octonion algebras over real closed fields as models of ACQ and ACO, respectively, even in extension arguments of this type.

\begin{theorem}[Model completeness]\label{thm:model-complete}
	The $L$-theories \textnormal{ACQ} and \textnormal{ACO} are model complete.
\end{theorem}

\begin{proof}
	We only prove the case of ACQ, the proof is exactly the same for ACO. Let $R_1,R_2\models \textnormal{RCF}$ such that $R_1\preceq R_2$ and consider $\hh_{R_1}\subseteq \hh_{R_2}$. We will prove $\hh_{R_1}\preceq \hh_{R_2}$.
	
	Let $\phi(\mathtt{q}_1,\dots,\mathtt{q}_n)$ be an $L_{\hh_{R_1}}$-formula and $q_1,\dots,q_n\in\hh_{R_1}$. By Proposition \ref{prop:bi-int}, there exists an $L_{R_1}$-formula $\psi(\mathtt{x}_{1,0},\dots,\mathtt{x}_{n,3})$ such that:
	\begin{equation}\label{eq:model-compl}
	\HH_{R_i}\models \phi(q_1,\dots,q_n)\text{ if and only if }R_i\models \psi(\jmath(q_1),\dots,\jmath(q_n)),
	\end{equation}
	where $\jmath:\hh_{R_i}\to R_i^4$ denotes the interpretation defined as $\jmath(q_\ell):=(x_{\ell,0},x_{\ell,1},x_{\ell,2},x_{\ell,3})\in R_i^4$, where $x_{\ell,0},x_{\ell,1},x_{\ell,2},x_{\ell,3}$ denote the unique elements of $R_i$ such that $q_\ell:=x_{\ell,0}+i x_{\ell,1}+j x_{\ell,2}+ k x_{\ell,3}$, for every $\ell\in\{1,\dots,n\}$ and $i=1,2$. By model completeness of the $L$-theory RCF, namely by Lemma \ref{lem:RCF}, we have that $R_1\models \psi(x_{1,0},\dots,x_{n,3})$ if and only if $R_2\models \psi(x_{1,0},\dots,x_{n,3})$ for every $x_{1,0},\dots,x_{n,3}\in R_1$. Then, we conclude by applying first (\ref{eq:model-compl}) with $i=1$, model completeness of RCF and  (\ref{eq:model-compl}) finally with $i=2$.
\end{proof}

\begin{corollary}[Completeness]\label{cor:complete}
	The $L$-theories \textnormal{ACQ} and \textnormal{ACO} are complete. In particular, they coincide with $\textnormal{Th}_L(\HH)$ and $\textnormal{Th}_L(\oo)$, respectively.
\end{corollary}

\begin{proof}
	Let $R\models \textnormal{RCF}$. Recall that $\overline{\Q}^r$, namely the real closure of $\Q$, is the smallest real closed field, so $\overline{\Q}^r\preceq R$. Hence, $\hh_{\overline{\Q}^r}\preceq\hh_{R}=\hh_{\overline{\Q}^r}\otimes_{\overline{\Q}^r} R$ and $\oo_{\overline{\Q}^r}\preceq\oo_{R}=\oo_{\overline{\Q}^r}\otimes_{\overline{\Q}^r} R$ by model completeness of ACQ and ACO, namely Theorem \ref{thm:model-complete}. Thus, we deduce that both the $L$-theories ACQ and ACO are complete since they have $\hh_{\overline{\Q}^r}$ and $\oo_{\overline{\Q}^r}$ as their minimal models.
\end{proof}

According to \cite[Theorem 1]{Ros78}, ACQ does not have quantifier elimination in the language of rings. We present an explicit counterexample in ACQ, the same $L$-formula disproves quantifier elimination for ACO in the language of rings as well. 

\begin{example}\label{ex:qe}
		Consider the following $L$-formula $\varphi$:
		\begin{equation}\label{eq:qe}
			\varphi(\mathtt{q}):=(\forall a\forall b((\mathtt{q}a)b=\mathtt{q}(ab)=(b\mathtt{q})a))\wedge \exists c\big((\forall a\forall b((ca)b=c(ab)=(bc)a))\wedge(c^2=\mathtt{q})\big).
		\end{equation}
		Let $\hh_R\models \textnormal{ACQ}$, with $R\models \textnormal{RCF}$ its center. Denote by $R_{\geq 0}:=\{x\in R\,|\, x\geq 0\}$. Observe that $\iota(R_{\geq 0}):=\{q\in\hh_R\,|\,\varphi(q)\}$, that is, $\iota(R_{\geq 0})$ is an $L$-definable subset of $\hh_R$, where $\iota:R\to\hh_R$ denotes the embedding of Proposition \ref{prop:bi-int}. Observe that, by Corollary \ref{cor:bi-int}, the class of quantifier-free $L$-definable subsets of $\hh_R$ can be regarded as belonging to the class of quantifier-free $L$-definable subsets of $R^4$, after applying the interpretation $\jmath:\hh_R\to R^4$. We prove that $\jmath\circ\iota(R_{\geq 0})=R_{\geq 0}\times\{0\}\times\{0\}\times\{0\}\subset R^4$ is not definable by a quantifier-free $L$-formula, hence a fortiori $\iota(R_{\geq 0})$ is so. Assume on the contrary that there exists a quantifier-free $L$-formula $\psi(\mathtt{x}_1,\mathtt{x}_2,\mathtt{x}_3,\mathtt{x}_4)$ such that $\jmath(\iota(R_{\geq 0}))=\{(x_1,\dots,x_4)\in R^4\,|\,\psi(x_1,\dots,x_4)\}$. Then, $R_{\geq 0}=\{x\in R\,|\,\psi(x,0,0,0)\}$ would be quantifier-free $L$-definable, but this is false since a subset of $R$ is quantifier-free $L$-definable if and only if it is empty, finite or cofinite. This proves that $\varphi(\mathtt{q})$ as above is an $L$-formula which is not elementarily equivalent modulo ACQ to any quantifier-free $L$-formula $\psi(\mathtt{q})$, as desired.
\end{example}

\begin{remark}\label{rem:conj}
	Consider the extensions $L_{\hh_R}:=L\cup\{q_i\}_{i\in \hh_R}$ and $L_{\oo_R}:=L\cup\{o_i\}_{i\in \oo_R}$ of $L$ with constant symbols from $\hh_R$ and $\oo_R$, respectively. Consider the conjugation function $^-_\hh:\hh_R\to\hh_R$ and $^-_\oo:\oo_R\to\oo_R$ defined as $\overline{q}_\hh:=a_0-a_1 i-a_2 j-a_3 k$ and $\overline{o}_\hh:=a_0-\sum_{h=1}^7 a_h e_h$ and the trace functions $\textnormal{tr}_{\hh}:\hh_R\to R$ and $\textnormal{tr}_{\oo}:\oo_R\to R$ defined as $\textnormal{tr}_\hh(q):=\frac{q+\overline{q}_\hh}{2}$ and $\textnormal{tr}_\oo(o):=\frac{o+\overline{o}_\hh}{2}$. Then, by Corollary \ref{cor:bi-int}, the functions $^-_\hh$, $\textnormal{tr}_{\hh}$ and $^-_\oo$, $\textnormal{tr}_{\oo}$ are $L_{\hh_R}$-definable and $L_{\oo_R}$-definable, respectively. Consequently, as in \cite{KT25}, the study of quantifier elimination for ACQ and ACO in the expanded languages $L'_{\hh_R}:=L_{\hh_R}\cup\{^-,\textnormal{tr}\}$ and $L'_{\oo_R}:=L_{\oo_R}\cup\{^-,\textnormal{tr}\}$, respectively, is strongly encouraged for future investigations.
\end{remark}

\vspace{1.5em}

\section{Ordered polynomials: the fragments of ordered formulas}\label{sec:3}

\vspace{0.5em}

\subsection{Algebraic sets, the Zariski topology \& real dimension} In this section we will consider a fragment of the theories ACQ and ACO. Let us introduce the set of $L$-terms we are admitting in this fragment.

\begin{definition}\label{def:ord-L-poly}
	Let $p$ be an $L$-polynomial in the variables $\mathtt{q}_1,\dots,\mathtt{q}_n$. We say that $p$ is an \emph{ordered $L$-polynomial} if it is given by the finite sum of monomials of the form
	\[
	\mathtt{q}_1^{\alpha_1}\cdot\mathtt{q}_2^{\alpha_2}\cdot\, \dots\,\cdot \mathtt{q}_n^{\alpha_n},
	\]
	for some $\alpha:=(\alpha_1,\dots,\alpha_n)\in\N^n$.
	
	We refer to the \emph{fragment of ordered $L$-formulas} as the fragment of $\textnormal{ACQ}$ or $\textnormal{ACO}$, respectively, in which the admitted $L$-terms are exactly ordered $L$-polynomials.
\end{definition}

Observe that Definition \ref{def:ord-L-poly} extends naturally when we admit coefficients in some fixed structure of ACQ or ACO.

\begin{definition}\label{def:ord-poly-coeff}
	Let $R\models\textnormal{RCF}$. Let $p$ be an $L_{\hh_R}$-polynomial or  an $L_{\oo_R}$-polynomial in the variables $\mathtt{q}_1,\dots,\mathtt{q}_n$. We say that $p$ is an \emph{ordered polynomial with coefficients in $\hh_R$ or in $\oo_R$}, respectively, if it is given by the finite sum of monomials of the form
	\[
	\mathtt{q}_1^{\alpha_1}\cdot\mathtt{q}_2^{\alpha_2}\cdot\, \dots\,\cdot \mathtt{q}_n^{\alpha_n} a_\alpha,
	\]
	for some $\alpha:=(\alpha_1,\dots,\alpha_n)\in\N^n$ and $a_\alpha\in \HH_R$ or $a_\alpha\in \oo_R$, respectively.
	
	We refer to the \emph{fragment of ordered formulas with coefficients in $\hh_R$ or in $\oo_R$}, respectively, as the fragments of $L_{\hh_R}$-formulas or $L_{\oo_R}$-formulas in which the admitted $L_{\hh_R}$-terms or the admitted $L_{\oo_R}$-terms are exactly ordered polynomials with coefficients in $\hh_R$ or $\oo_R$, respectively.
\end{definition}

There are at least two motivations for studying the fragment of ordered $L$-formulas in ACQ and ACO. The first motivation is related to hypercomplex analysis and the notion of slice regular functions already mentioned in the Introduction. The second motivation to study the fragment of ordered $L$-formulas in ACQ and ACO concerns our Proposition \ref{prop:bi-int} and Corollary \ref{cor:bi-int}. Indeed, by bi-interpreting the theories ACQ and RCF and the theories ACO and RCF we have shown that for every $R\models\textnormal{RCF}$ the class of definable subsets with parameters in $\hh_R^n$ and $\oo_R^n$ coincides (via the bi-interpretation) with the class of semialgebraic subsets of $R^{4n}$ and $R^{8n}$, respectively. Observe that this procedure may naturally involve terms consisting of non-ordered polynomials. Thus, the classes of algebraic sets defined by ordered polynomials and of definable sets by means of ordered polynomials in $\hh_R^n$ or $\oo_R^n$ constitute subclasses of those sets studied in real algebraic geometry. This is why the development of proper quaternionic and octonionic techniques is of special interest.

\begin{definition}\label{def:algset}
	Let $X\subset\hh_R^n$ or $X\subset\oo_R^n$. We say that $X$ is a \emph{basic algebraic subset of $\hh_R^n$ or of $\oo_R^n$}, respectively, if $X$ can be described as the common solution set of a finite number of ordered polynomial equations with coefficients in $\hh_R^n$ or $\oo_R^n$, respectively. In other words, $X$ is a \emph{basic algebraic subset of $\hh_R^n$ or of $\oo_R^n$}, respectively, if
	\begin{align*}
		X&:=\{(q_1,\dots,q_n)\in\hh_{R}^n\,|\,\bigwedge_{s=1}^\ell p_s(q_1,\dots,q_n)=0\}\quad or\\ X&:=\{(o_1,\dots,o_n)\in\oo_{R}^n\,|\,\bigwedge_{s=1}^\ell p_s(o_1,\dots,o_n)=0\},
	\end{align*}
	for some ordered polynomials $p_s(\mathtt{q}_1,\dots,\mathtt{q}_n)$ with coefficients in $\hh_R$ or in $\oo_R$, for every $s\in\{1,\dots,\ell\}$, respectively.
	
	Let $X\subset\hh_R^n$ or $X\subset\oo_R^n$. We say that $X$ is an \emph{algebraic subset of $\hh_R^n$ or of $\oo_R^n$}, respectively, if $X$ is a finite union of basic algebraic subsets of $\hh_R^n$ or $\oo_{R}^n$, respectively.
\end{definition}

Observe that the family of algebraic subsets of $\hh_R^n$ (or $\oo_R^n$, respectively) constitute the closed sets of a topology of $\hh_R^n$ (or $\oo_R^n$, respectively). More precisely, this is the coarsest topology of $\hh_R^n$ (or $\oo_R^n$, respectively) whose family of closed sets contains the zero loci of ordered polynomials. We call this topology the \emph{Zariski topology of $\hh_R^n$} (or \emph{Zariski topology of $\oo_R^n$}, respectively).
		
Indeed, to show this is the case we only have to prove that if $\{X_i\}_{i\in I}\subset\hh_R^n$ (or $\{X_i\}_{i\in I}\subset\oo_R^n$, respectively) is a family of algebraic subsets, then $\bigcap_{i\in I}X_i$ is an algebraic subset of $\hh_R^n$ (or $\oo_R^n$, respectively). We just deal with the quaternionic case, a similar argument works for octonions as well. Consider the interpretation $\jmath:\hh_R\to R^4$ and the real algebraic set $\jmath^n(X_i)\subset R^{4n}$ for every $i\in I$. Then, the Noetherianity of the real Zariski topology of $R^{4n}$ ensures the existence of $N\in\N$ and $i_1,\dots,i_N\in I$ such that 
\[
	\bigcap_{i\in I}\jmath^n(X_i)=\bigcap_{s=1}^N\jmath^n(X_{i_s})\subset R^{4n}.
\]
Hence, $\bigcap_{i\in I}X_i=\bigcap_{s=1}^N X_{i_s}$. Write each $X_{i_s}$ as a finite union of basic algebraic subsets of $\hh_R^n$, that is, $X_{i_s}=\bigcup_{j=1}^{m_i} Y_j^{(i_s)}$ for some $m_i\in\N$ and basic algebraic subsets $Y_{j}^{(i_s)}\subset\hh_R^n$ for every $j\in J_i:=\{1,\dots,m_i\}$. Then, $\bigcap_{i\in I}X_i=\bigcap_{s=1}^{N} \bigcup_{j=1}^{m_i} Y_j^{(i_s)}=\bigcup_{(j_1,\dots,j_N)\in J} Y_{(j_1,\dots,j_N)}$, where $J:=J_1\times\dots\times J_N$ is a finite set and $Y_{(j_1,\dots,j_N)}:=\bigcap_{s=1}^{N} Y_{j_s}^{(i_s)}$ is a basic algebraic subset of $\hh_{R}^n$ for every $(j_1,\dots,j_N)\in J$. This shows that $\bigcap_{i\in I}X_i$ is an algebraic subset of $\hh_{R}^n$.

Since the Zariski topology of $\hh_R^n $ (or $\oo_R^n$, respectively) is (strictly) coarser than the topology induced by the usual real Zariski topology via the interpretation $\jmath:\hh_R\to R^4$ (or $\jmath:\oo_R\to R^8$, respectively), then the Zariski topology of $\hh_R^n $ (or $\oo_R^n$, respectively) is Noetherian as well. Hence, we dispose of the notions of \emph{irreducible algebraic set}, \emph{irreducible components} of an algebraic set and \emph{irreducible decomposition} of an algebraic set both over $\hh_{R}$ and $\oo_R$. Observe that, by definition, every irreducible algebraic set $X\subset\hh_R^n$ (or $\oo_R^n$, respectively) is a basic algebraic set.

\begin{remark}\label{rmk:basic-alg}
	There are algebraic subsets of $\hh_R^n$ and $\oo_R^n$ which cannot be expressed as basic algebraic ones. We just state our example over $\hh_R$ but it holds over $\oo_R$ as well. Let $X:=\{i,-i\}\subset\hh_R^n$. Clearly, $X$ is algebraic according to Definition \ref{def:algset}, but it is not a basic algebraic set. Indeed, each ordered polynomial $p(\mathtt{q})$ with coefficients in $\hh_R$ vanishing at $X$ must vanish at the whole basic algebraic set
	\begin{equation}\label{eq:basic}
		Y:=\{q\in\hh_R\,|\,q^2+1=0\}\supsetneq \{i,-i\}=:X.
	\end{equation}
	In particular, this shows that the ideals $\mathcal{I}(X)$ and $\mathcal{I}(Y)$ of ordered polynomials vanishing on $X$ and $Y$, respectively, as defined and studied in \cite{GSV25}, do coincide, whereas the two algebraic sets $X,Y\subset\hh_R$ have very different geometric properties ($X$ is reducible and $\jmath(X)\subset R^4$ has real dimension $0$ while $Y$ is irreducible and $\jmath(Y)\subset R^4$ has real dimension $2$).
	
\end{remark}

As a consequence of model completeness for ACQ and ACO, next definition is well defined, in the sense that it does not depend on the choice of the polynomials $p_1,\dots,p_\ell$ defining $X\subset\hh_R^n$ or $X\subset\oo_{R}^n$, respectively, but just on the set itself.

\begin{definition}\label{def:algset-coeff}
	Let $R_1, R_2\models\textnormal{RCF}$ such that $R_1\preceq R_2$. Denote by $q:=(q_1,\dots,q_n)$ and $o:=(o_1,\dots,o_n)$. Let $X=\{q\in\hh_{R_1}^n\,|\,\bigwedge_{s=1}^\ell p_s(q)=0\}$ or $X=\{o\in\oo_{R_1}^n\,|\,\bigwedge_{s=1}^\ell p_s(o)=0\}$ be an algebraic subset of $\hh_{R_1}^n$ or $\oo_{R_1}^n$, respectively. Denote by $X_{\hh_{R_2}}:=\{q\in\hh_{R_2}^n\,|\,\bigwedge_{s=1}^\ell p_s(q)=0\}\subset \hh_{R_2}^n$ the \emph{extension of $X$ to $\hh_{R_2}$} or $X_{\oo_{R_2}}=\{o\in\oo_{R_2}^n\,|\,\bigwedge_{s=1}^\ell p_s(o)=0\}\subset \oo_{R_2}^n$ the \emph{extension of $X$ to $\oo_{R_2}$}, respectively.
\end{definition}

Let $R\models\textnormal{RCF}$. Let $p(\mathtt{q}_1,\dots,\mathtt{q}_n)$ be an ordered polynomial with coefficients in $\hh_R$ and $p'(\mathtt{o}_1,\dots,\mathtt{o}_n,)$ be an ordered polynomial with coefficients in $\oo_R$. Denote by
\begin{align*}
	\mathcal{Z}_{\hh_R}(p)&:=\{(q_1,\dots,q_n)\in\hh_R^n\,|\,p(q_1,\dots,q_n)=0\}\\
	\mathcal{Z}_{\oo_R}(p')&:=\{(o_1,\dots,o_n)\in\oo_R^n\,|\,p'(o_1,\dots,o_n)=0\}
\end{align*}
the quaternion and octonion zero loci of $p$ and $p'$, respectively. 

\begin{remark}
	Observe that, by Corollary \ref{cor:bi-int}, both $\mathcal{Z}_{\hh_{R}}(p)\subset\hh_R^n$ and $\mathcal{Z}_{\oo_R}(p')\subset\oo_R^n$ can be interpreted in $R^{4n}$ and $R^{8n}$, respectively, as real algebraic sets.
\end{remark}

Let us generalize some properties of ordered polynomials with coefficients in $\hh_R$ and $\oo_R$ originally proved in \cite{Niv41,GP22} in the case $R=\R$. For the one-variable case we refer the interested reader to \cite[Ch. 5, \S 16]{Lam01} for a general treatment for non-commutative associative rings and \cite{GS08,GP11} for generalizations to slice regular functions.

Let us introduce a first-order characterization of the notion of dimension for real algebraic sets. Let $X\subset R^n$ be an algebraic set. Denote by $\mathcal{I}(X)\subset R[\mathtt{x}_1,\dots,\mathtt{x_n}]$ the ideal of polynomials vanishing on $X$. Recall that, by definition, $\dim_R(X)$ is the Krull dimension of the ring $R[\mathtt{x}_1,\dots,\mathtt{x_n}]/\mathcal{I}(X)$.

\begin{lemma}\label{lem:alg-dim}
	Fix $L_{R}:=L\cup\{c_i\}_{i\in R}$ to be the extension of $L$ with constant symbols taken from $R$. Let $X\subset R^n$ be an algebraic set, then $\dim_R(X)$ is $L_{R}$-definable.
\end{lemma}

\begin{proof}
	Let $X_1,\dots,X_\ell$ be the irreducible components of $X$. By \cite[Theorem 2.8.3(i)]{BCR98}, we have $\dim_R(X)=\max(\dim_R(X_1),\dots,\dim_R(X_\ell))$. Observe that, for every $a,b\in R$ with $a\neq b$, $a<b$ if and only if $\exists c(b-a=c^2)$, thus it is possible to describe the maximum of a finite set without using the symbol $<$. Thus, we are only left to prove that $\dim_R(X)$ is $L$-definable for an irreducible algebraic set $X\subset R^n$.
	
	Let $X\subset R^n$ be an irreducible algebraic set and let $\mathcal{I}(X)=(p_1,\dots,p_\ell)$. Denote by $\textnormal{Reg}(X):=\{x\in X\,|\,\textnormal{rk}([\frac{\partial p_s}{\partial \mathtt{x}_t}(x)]_{s=1,\dots,\ell,\,t=1,\dots,n})=n-\dim_R(X)\}$ the set of nonsingular points of $X$ (see \cite[Definition 3.3.4]{BCR98}). Observe that $\textnormal{Reg}(X)\subset X$ is a non-empty Zariski open set, hence Zariski dense in $X$ by irreducibility of $X$. In addition, \cite[Proposition 3.3.2]{BCR98} shows that we have an equivalent definition of $\dim_R(X)$ in terms of nonsingular points, that is:
	\[
	\dim_R(X)=\max\Big(\Big\{d\in\{1,\dots,n\}\,\Big|\,\exists x\in X \,\Big(\textnormal{rk}\Big(\Big[\frac{\partial p_s}{\partial \mathtt{x}_t}(x)\Big]_{s=1,\dots,\ell,\,t=1,\dots,n}\Big)= n-d\Big)\Big\}\Big),
	\]
	where $\mathtt{x}:=(\mathtt{x}_1,\dots,\mathtt{x}_n)$.
	Observe that the latter equivalent definition of $\dim_R(X)$, for an irreducible algebraic set $X\subset R^n$, is $L_{R}$-definable since the rank of a matrix can be described in terms of vanishing minors, as desired.
\end{proof}

\begin{remark}\label{rem:Q-coeff}
	Assume $X\subset R^n$ is a $\Q$-algebraic set, that is, $X$ is an algebraic set that can be described by polynomial equations without coefficients. By considering the $\Q$-irreducible components of $X$ (see \cite[Definition 2.1.5]{FG}) we may suppose $X$ is $\Q$-irreducible as in Lemma \ref{lem:alg-dim}. Then, the thesis of Lemma \ref{lem:alg-dim} can be refined by stating that $\dim_R(X)$ is $L$-definable even though, in general, 
	\[
	\mathcal{I}(X)=\mathcal{I}_{\overline{\Q}^r}(X)R[\mathtt{x}]\supsetneq \mathcal{I}_{\Q}(X)R[\mathtt{x}],
	\]
	where $\mathcal{I}_{K}(X):=\mathcal{I}(X)\cap K[\mathtt{x}]$ denotes the ideal of polynomials with coefficients in $K$ vanishing over $X$, for $K=\Q,\overline{\Q}^r$ (see \cite[Theorem 3.1.2]{FG}). Indeed, let $\mathcal{I}_\Q(X)=(q_1,\dots,q_\ell)$, for some $q_1(\mathtt{x}),\dots,q_\ell(\mathtt{x})\in\Q[\mathtt{x}]$. As above, we may reduce to the case where $X\subset R^n$ is $\Q$-irreducible, that is, $\mathcal{I}_\Q(X)$ is a prime ideal of $\Q[\mathtt{x}]$. In \cite[Definition 4.1.1]{FG}, Fernando and Ghiloni introduce the notion of $R|\Q$-regular points and prove, in \cite[Proposition 5.2.6 \& Corollary 5.2.8]{FG}, that
	\[
	\textnormal{Reg}^*(X):=\Big\{x\in X\,\Big|\,\textnormal{rk}\Big(\Big[\frac{\partial q_s}{\partial \mathtt{x}_t}(x)\Big]_{s=1,\dots,\ell,\,t=1,\dots,n}\Big)=n-\dim_R(X)\Big\}\subset\textnormal{Reg}(X)
	\]
	is a non-empty $\Q$-Zariski open subset of $X$. Hence, if $X\subset R^n$ is a $\Q$-algebraic set, then
	\[
	\dim_R(X)=\max\Big(\Big\{d\in\{1,\dots,n\}\,\Big|\,\forall\mathtt{x}\,\Big(\textnormal{rk}\Big(\Big[\frac{\partial q_s}{\partial \mathtt{x}_t}(\mathtt{x})\Big]_{s=1,\dots,\ell,\,t=1,\dots,n}\Big)\leq n-d\Big)\Big\}\Big),
	\]
	where $\mathtt{x}:=(\mathtt{x}_1,\dots,\mathtt{x}_n)$.
	Since vanishing conditions on minors are equations over $\Q$, we deduce that $\dim_R(X)$ is $L$-definable, as desired. For more details about real $\Q$-algebraic geometry we refer to \cite{FG,GSa,Sav1}.
\end{remark}

\begin{definition}
	Let $R\models\textnormal{RCF}$. Let $X\subset\hh_R^n$ or $X\subset\oo_R^n$, we define the \emph{real dimension} $\dim_R(X)$ of $X$ as $\dim_R(\jmath^n(X))$ of $\jmath^n(X)\subset R^{4n}$ or $\jmath^n(X)\subset R^{8n}$, respectively.
\end{definition}

\begin{corollary}\label{cor:ext-dim}
	Let $R_1, R_2\models\textnormal{RCF}$ such that $R_1\preceq R_2$. Let $X\subset \hh_{R_1}^n$ or $X\subset\oo_{R_1}^n$ be an algebraic set. Then, $\dim_{R_1}(X)=\dim_{R_2}(X_{R_2})$.
\end{corollary}

\begin{proof}
	Let $L_{\hh_{R_1}}:=L\cup\{c\}_{c\in\hh_{R_1}}$ and $L_{\oo_{R_1}}:=L\cup\{c\}_{c\in\oo_{R_1}}$ denote the language $L$ extended with coefficients in $\hh_{R_1}$ or in $\oo_{R_1}$, respectively. Consider the real algebraic subset $\jmath^n(X)$ of $R_1^{4n}$ or $R_1^{8n}$, respectively. Then, by Lemma \ref{lem:alg-dim}, $\dim_{R_1}(\jmath^n(X))=d$ can be expressed as an $L_{R_1}$-statement. By Corollary \ref{cor:bi-int} and Remark \ref{rmk:L'}, we obtain that $\dim_{R_1}(\jmath^n(X))=d$ can be expressed as an $L_{\hh_{R_1}}$-statement or as an $L_{\oo_{R_1}}$-statement satisfied by $X$, respectively. Then, by model completeness of ACQ and ACO, namely by Theorem \ref{thm:model-complete}, we obtain $\dim_{R_1}(X)=\dim_{R_2}(X_{R_2})$, as desired.
\end{proof}

\vspace{0.5em}

\subsection{Properties of ordered polynomials}

Let $R\models\textnormal{RCF}$. Let us deduce some properties of the zero loci of ordered polynomials by model completeness of ACQ and ACO. At first, we prove the Fundamental Theorem of Algebra for ordered polynomials over every model of ACQ and ACO.

\begin{notation}
	Let $R\models\textnormal{RCF}$. In this subsection we will identify the structures $\langle R_1^4, +,-,*,\underline{0},\underline{1} \rangle$ with $\langle \hh_{R}, +,-,\cdot,0,1 \rangle$ and $\langle R_1^8, +,-,*,\underline{0},\underline{1} \rangle$ with $\langle \oo_{R}, +,-,\cdot,0,1 \rangle$  in terms of Proposition \ref{prop:bi-int}. We will refer to real coordinates when interpreting semialgebraic subsets of $R^{4n}$ or $R^{8n}$ as subsets of $\hh_R^n$ or $\oo_R^n$ in terms of Corollary \ref{cor:bi-int}.
\end{notation}

We point out that next result was already known in the literature for every model of ACQ, see \cite[\S 16]{Lam01}, proven by Niven with a different strategy. However, the result was not known for models of ACO and our proof for ACQ extends to ACO.

\begin{theorem}[Fundamental Theorem of Algebra]\label{thm:FTA}
	Let $p(\mathtt{q})$ and $p'(\mathtt{o})$ be non-constant ordered polynomials with coefficients in $\hh_R$ and $\oo_R$, respectively. Then, $\mathcal{Z}_{\hh_R}(p)\subset\hh_R$ is a non-empty finite union of isolated points and non-intersecting $2$-spheres centered at real points and $\mathcal{Z}_{\oo_R}(p')\subset\oo_R$ is a finite union of isolated points and non-intersecting $6$-spheres centered at a real point.
\end{theorem}

\begin{proof}
	Let us prove that $\mathcal{Z}(p)\subset\hh_{R}$ is a non-empty finite union of isolated points and disjoint spheres $\mathbb{S}^{2}\subset\hh_{R}$, the proof for $\mathcal{Z}(p')\subset\oo_R$ is similar. Observe that the statement \emph{``Every polynomial of positive degree $\leq d$ has a root"} can be expressed as a first-order $L$-sentence $\rho_d$ for every $d\in \N$. Since $\hh_{\R}\models\rho_d$ for every $d\in \N$ by \cite{Niv41}, completeness of ACQ, namely Corollary \ref{cor:complete}, ensures that $\hh_R\models\rho_d$ for every $d\in \N$ as well. Hence, $\mathcal{Z}(p)\neq\emptyset$.
	
	Let $\beta_d(\mathtt{q})$, $\sigma_d(\mathtt{q})$ and $\tau_d$ denote the following first-order $L_{\hh_{R}}$-formulas:
	
	\begin{align*}
		\beta_d(\mathtt{q},a)&:\,\exists r\in R_{>0}\,\forall q \Big(\Big((q\in B_R(\mathtt{q},r))\vee q\neq\mathtt{q}\Big) \rightarrow \sum_{i=0}^d q^i a_i\neq 0\Big),\\
		\sigma_d(\mathtt{q},a)&:\,\exists c\in R,\,\exists r\in R_{>0}\,\Big(\mathtt{q}\in S_R(c,r)\wedge\forall q \Big((q\in S_R(c,r)) \rightarrow \sum_{i=0}^d q^i a_i = 0\Big)\wedge\\
		&\exists r_1, r_2\in R_{>0}(r\in(r_1,r_2))\wedge\Big((q\in B_R(c,r_2)\cap (B_R(c,r_1)\cup S_R(c,r))^c)\rightarrow (\sum_{i=0}^d q^i a_i \neq 0)\Big)\Big),\\
		\tau_d(a)&:\,\forall c_1,c_2\in R\,\, \forall r_1,r_2\in R_{>0}\,\Big(\\
		&\Big(\forall q ((q\in S_R(c_1,r_1)\cup S_R(c_2,r_2)) \rightarrow \sum_{i=0}^d q^i a_i = 0)\Big)\rightarrow S_R(c_1,r_1)\cap S_R(c_2,r_2)=\emptyset\Big),\\
	\end{align*}
	where $a:=(a_1,\dots,a_d)\in \hh_R^d$, $R_{>0}:=\{q\in\hh_R\,|\, q\neq 0, \exists p(p\in R)\wedge(q=p^2)\}$, $B_R(\mathtt{q},r)$ denotes the ball in $\hh_R$ of center $\mathtt{q}$ and radius $r\in R_{>0}$ and $S_R(c,r)$ the sphere in $\hh_R$ of center $c$ and radius $r\in R_{>0}$. Define the $L_{\hh_{\overline{\Q}^r}}$-sentence $\phi_d$ as follows: 
	\begin{align*}
		\phi_d:\,\forall a_0\dots\forall a_d \forall \mathtt{q}\Big(\big(\sum_{i=0}^d \mathtt{q}^i a_i=0\big)\rightarrow \big(\beta_d(\mathtt{q},a)\vee (\sigma_d(\mathtt{q},a)\wedge\tau_d(a))\big) \Big).
	\end{align*}
	Observe that $\phi_d$ is the sentence expressing the following statement: \emph{``The zero locus of each ordered polynomial of degree $\leq d$ consists of a finite number of isolated points and disjoint spheres centered at real values.".} Observe that the number of points and spheres is finite because the zero locus $\mathcal{Z}(p)\subset\hh_R$ of $p$ is also an algebraic set in real coordinates, so it has a finite number of real irreducible components, see \cite[Theorem 2.8.3]{BCR98}. Observe that $\hh\models\phi_d$ for every $d\in\N$ by \cite[Theorem 1.3]{GSV08} and $\phi_d$ is an $L_{\hh_{\overline{\Q}^r}}$-sentence by Corollary \ref{cor:bi-int} and Remark \ref{rem:Q-coeff}. Thus, completeness of ACQ, namely Corollary \ref{cor:complete}, ensures that $\hh_R\models\phi_d$ for every $d\in \N$ and for every $R\models \textnormal{RCF}$, as desired.
\end{proof}

\begin{remark}
	Observe that the formula $\beta_d(\mathtt{q},a)\vee (\sigma_d(\mathtt{q},a)\wedge\tau_d(a)$ has coefficients $a=(a_1,\dots,a_d)$ in $\hh_R$. When considering the sentence $\phi_d$ no coefficients from $\hh_R$ are needed but, in order to represent with the language of rings the semialgebraic sets $B_R(\mathtt{q},r)$ we will need to add constants $i,j,k$, see Corollary \ref{cor:bi-int}. That is the reason why $\psi_d$ turns out to be an $L_{\hh_{\overline{\Q}^r}}$-sentence.
\end{remark}

Observe that there are non-ordered polynomials with coefficients in $\hh_R$ and $\oo_R$ having empty zero locus, as the next example shows. 

\begin{example}
	Consider the polynomial $p(\mathtt{q}):=\frac{1}{16}(\mathtt{q}-i\mathtt{q}i-j\mathtt{q}j-k\mathtt{q}k)^2+1$. In the real coordinates of $\hh_R$ the polynomial $p(\mathtt{q})$ corresponds to $\mathtt{x}_0^2+1$, which has no real roots in $R^4$.
\end{example}

\begin{remark}\label{rmk:zero-set}
	The class of algebraic sets as in Definitions \ref{def:algset}\&\ref{def:algset-coeff} is strictly smaller than the class of sets that can be defined by finite systems of non-ordered polynomial equations. Consider $\hh_R$, a similar argument works for $\oo_{R}$. Consider the set $$X:=\Big\{q=x_0+ix_1+jx_2+kx_3\in\hh_R\,\Big|\,x_{0}= \frac{1}{4}\Big(q-iqi-jqj-kqk\Big)=0\Big\}$$ defined by a non-ordered polynomial equation. Observe that its real dimension is $3$ and it is not compact in the Euclidean topology of $\hh_R$. However, as a consequence of Theorem \ref{thm:FTA}, algebraic subsets of $\hh_R$ have real dimension $\leq 2$ and they are compact with respect to the euclidean topology of  $\hh_R$.
\end{remark}

Here we propose an extension of \cite[Propositions 3.15 \& 3.17]{GP22} for every model of ACQ and ACO. Again, our proof is a direct consequence of completeness of ACQ and ACO and Corollary \ref{cor:ext-dim}.

\begin{proposition}\label{prop:real-dim}
	Let $R\models\textnormal{RCF}$. Let $p(\mathtt{q}_1,\dots,\mathtt{q}_n)$ and $p'(\mathtt{o}_1,\dots,\mathtt{o}_n)$ be non-constant ordered polynomials with coefficients in $\hh_R$ and in $\oo_R$, respectively. Then
	\[
	4(n-1)\leq\dim_R(\mathcal{Z}_{\hh_{R}}(p))\leq 4n-2
	\]
	and
	\[
	 8(n-1)\leq\dim_R(\mathcal{Z}_{\oo_{R}}(p'))\leq 8n-2.
	\]
\end{proposition}

\begin{proof}
	We prove the quaternionic case, the same strategy works for octonions as well. By Corollary \ref{cor:bi-int} and Lemma \ref{lem:alg-dim}, the dimension of an algebraic subset of $\hh_R^n$ is $L_{\hh_R}$-definable and, if the starting algebraic set is defined by $L$-polynomials, then its dimension is $L$-definable by Remark \ref{rem:Q-coeff}. Let $p(\mathtt{q}):=\sum_{|\alpha|\leq d} \mathtt{q}_1^{\alpha_1}\dots \mathtt{q}_n^{\alpha_n} a_\alpha$ where $\mathtt{q}:=(\mathtt{q}_1,\dots,\mathtt{q}_n)$, $\alpha\in\N^n$ and $a_\alpha\in\hh_{R}$. Let $\psi_k(a)$ be the first order $L_{\hh_R}$-formula, depending on the string of coefficients $a$ of $p(\mathtt{q})$, stating: \emph{``The zero locus of the ordered polynomial $p(\mathtt{q})$ has real dimension $k$."}.
	
	Let $\Lambda:=\{m\in\N^n\,|\,|m|\leq d\}$ and let $a:=(a_{\alpha})_{\alpha\in\Lambda}$ be a string of elements of $\hh_R$. Consider the $L$-sentence $\phi_d$ defined as:
	\[
	\phi_d:\quad \forall a \Big(\psi_{4(n-1)}(a)\vee\psi_{4n-3}(a)\vee\psi_{4n-2}(a)\Big).
	\]
	Observe that $\phi_d$ is an $L$-sentence stating: \emph{``Every ordered polynomial $p(\mathtt{q})$ of degree $\leq d$ satisfies $4(n-1)\leq\dim_R(\mathcal{Z}_{\hh_{R}}(p))\leq 4n-2$."}. By \cite[Propositions 3.15 \& 3.17]{GP22} we have $\hh\models\phi_d$ for every $d\in \N$, thus completeness of ACQ, namely Corollary \ref{cor:complete}, ensures that $\hh_R\models\phi_d$ for every $d\in \N$ and for every $R\models \textnormal{RCF}$, as desired.
\end{proof}

\begin{remark}
	The previous estimate of the real dimension of the zero loci of ordered polynomials is sharp for $\hh_R$, that is, for every $n\geq 2$ there are ordered polynomials whose zero loci have real dimensions $4(n-1),\,4n -3$ and $4n-2$. Consider the ordered polynomials $p_1(\mathtt{q}):=\mathtt{q}_1$, $p_2(\mathtt{q}):=\mathtt{q}_1^2+\mathtt{q}_2^2+1$ and $p_3(\mathtt{q}):=\mathtt{q}_1^2+1$, with $\mathtt{q}:=(\mathtt{q}_1,\dots,\mathtt{q}_n)$ for $n\geq 2$. If $R=\R$, by \cite[Example 3.16]{GP22} we have that $\dim_\R(\mathcal{Z}_{\hh_\R}(p_1))=4(n-1)$, $\dim_\R((\mathcal{Z}_{\hh_\R}(p_2))=4n-3$ and $\dim_\R(\mathcal{Z}_{\hh_\R}(p_3))=4n-2$. Observe that, by Corollary \ref{cor:ext-dim}, the ordered polynomials $p_1(\mathtt{q})$, $p_2(\mathtt{q})$ and $p_3(\mathtt{q})$ satisfy $\dim_R(\mathcal{Z}_{\hh_R}(p_1))=4(n-1)$, $\dim_R(\mathcal{Z}_{\hh_R}(p_2))=4n-3$ and $\dim_R(\mathcal{Z}_{\hh_R}(p_3))=4n-2$ for every real closed field $R$, as desired.
\end{remark}

\vspace{0.5em}

\subsection{Failure of Quantifier Elimination for the fragment of ordered formulas} This last section is devoted to studying the fragment of ACQ and ACO consisting of those $L$-formulas that can be written only by means of ordered $L$-terms. By \emph{ordered $L$-term} we mean an $L$-term in which only ordered $L$-polynomials are involved. We refer to the \emph{fragment of ordered $L$-formulas} as the fragments of ACQ and ACO in which only ordered $L$-terms occur. In principle, since the class of ordered polynomials have much stricter algebraic properties (see Theorem \ref{thm:FTA} and Proposition \ref{prop:real-dim}) it is worth studying properties of this subclass of $L$-formulas. Unfortunately, quantifier elimination does not hold for the fragment of ordered $L$-formulas, as expressed by the next result. 

\begin{theorem}\label{thm:ord-form}
	Every $L$-formula is elementarily equivalent modulo ACQ and ACO to an ordered $L$-formula. Equivalently, every $L$-definable subset of $\hh_R^n$ and $\oo_R^n$ is definable by means of an ordered $L$-formula.
\end{theorem}

\begin{proof}
	By induction on first-order $L$-formulas, it suffices to prove the statement for atomic $L$-formulas. In the case of atomic $L$-formulas, we will prove the following stronger version of the statement: \emph{Every $L$-formula $p(\mathtt{o})=0$ is elementarily equivalent modulo \textnormal{ACQ} (or \textnormal{ACO}, respectively) to an $L$-formula of the form $\exists t (\bigwedge_{i=1}^e (p_i(\mathtt{o},t)=0))$, where $\mathtt{t}$ denotes a string of new variables and each $p_i(\mathtt{o},\mathtt{t})$ is an ordered $L$-polynomial}.
	
	We just prove the latter statement in the octonionic case, the same procedure, suitably simplified, works for quaternions as well. Let $p(\mathtt{o})$ be an $L$-polynomial and consider the atomic $L$-formula $p(\mathtt{o})=0$. Let us argue by induction on the number $a$ of non-ordered $L$-monomials defining $p(\mathtt{o})$, that is, on the number of $L$-monomials of $p(\mathtt{o})$ that are not of the form $\mathtt{o}_1^{\alpha_1}\mathtt{o}_2^{\alpha_2}\dots \mathtt{o}_n^{\alpha_n}$, for some $\alpha:=(\alpha_1,\dots,\alpha_n)\in\N^n$. If $a=0$, then $p(\mathtt{o})$ is an ordered $L$-polynomial and there is nothing to prove. Suppose $p(\mathtt{o})$ has $a>0$ non-ordered $L$-monomials. Observe that such $L$-monomials are of the following form:
	\begin{equation}\label{eq:nonordered-monomial}
		(\dots(\mathtt{o}_{b_1}^{\beta_{b_1}}\dots (\dots (\mathtt{o}_{b_i}^{\beta_{b_i}}\dots\mathtt{o}_{b_{i+k}}^{\beta_{b_{i+k}}})\dots )\dots \mathtt{o}_{b_\ell}^{\beta_{b_\ell}})\dots )
	\end{equation}
	for some $\ell\in\N^*$, $\beta_{b_1},\beta_{b_2},\dots,\beta_{b_\ell}\in \N$, $b_1,b_2\dots,b_\ell\in\{1,\dots,n\}$ and for some choice of parentheses indicating when the order on the products of the variables differs from the usual ``from left to right" product. Our strategy by induction works as follows: first we eliminate the parenthesis and then we order the variables. Both procedures will require the introduction of extra variables and existential quantifiers.
	
	Let $m(\mathtt{o})$ be an $L$-monomial of $p(\mathtt{o})$ as in (\ref{eq:nonordered-monomial}). If no parenthesis occur in $m(\mathtt{o})$, that is, if the $L$-monomial $m(\mathtt{o})$ is of the form $\mathtt{o}_{b_1}^{\beta_{b_1}}\mathtt{o}_{b_2}^{\beta_{b_2}} \dots\mathtt{o}_{b_\ell}^{\beta_{b_\ell}}$, we proceed inductively by ordering the variables. Since $a>0$ and $m(\mathtt{o})$ is a non-ordered monomial, there is $h\in\{1,\dots,\ell-1\}$ such that $b_{k}<b_{k+1}$ for every $k\in\{1,\dots,h-1\}$ and $b_{h}>b_{h+1}$. Consider the ordered $L$-monomial $m'(\mathtt{o},\mathtt{t})$ defined as 
	\begin{equation}\label{eq:ordered-monomial}
		\mathtt{o}_{b_1}^{\beta_{b_1}}\mathtt{o}_{b_2}^{\beta_{b_2}}\dots \mathtt{o}_{b_h}^{\beta_{b_h}}\mathtt{t}_{1}^{\beta_{b_{h+1}}} \mathtt{t}_{2}^{\beta_{b_{h+2}}}\dots\mathtt{t}_{\ell-h'}^{\beta_{b_\ell}}
	\end{equation}
	in the ordered variables $\mathtt{o}_1,\dots,\mathtt{o_{n}},\mathtt{t}_1,\dots,\mathtt{t}_{\ell-h'}$ and define $p'(\mathtt{o},\mathtt{t})$ as the $L$-polynomial obtained by substituting the $L$-monomial $m(\mathtt{o})$ in $p(\mathtt{o})$ with the ordered $L$-monomial $m'(\mathtt{o},\mathtt{t})$ in (\ref{eq:ordered-monomial}). By construction, $p'(\mathtt{o},\mathtt{t})$ has exactly $(a-1)$ non-ordered monomials. Consider the ordered $L$-formula $\phi(\mathtt{o},\mathtt{t})$ defined as
	\[
	(p'(\mathtt{o},\mathtt{t})=0)\wedge\Big(\bigwedge_{k=1}^{\ell-h}(\mathtt{t}_{k}-\mathtt{o}_{b_{k+h}}=0)\Big).
	\]
	Observe that the $L$-formulas $p(\mathtt{o})=0$ and $\exists t \phi(\mathtt{o},t)$ are elementarily equivalent modulo ACO and each $L$-polynomial $\mathtt{t}_{k}-\mathtt{o}_{b_{k+h}}$ is ordered. Then, by inductive assumption on the number $(a-1)$ of non-ordered $L$-monomials occurring in the polynomial $p'(\mathtt{o},\mathtt{t})$, we get an $L$-formula $\phi'(\mathtt{o},\mathtt{t})$ of the form $\exists t' (\bigwedge_{i=1}^e (p'_i(\mathtt{o},\mathtt{t},t')=0))$ for some $e\in\N$, some string of new variables $\mathtt{t}'$ and for some ordered $L$-polynomials $p'_i(\mathtt{o},\mathtt{t},\mathtt{t}')$ for every $i\in\{1,\dots,e\}$, such that $p'(\mathtt{o},\mathtt{t})=0$ and $\phi'(\mathtt{o},\mathtt{t})$ are elementarily equivalent modulo ACO. Hence, the $L$-formulas $\Phi(\mathtt{o})$ defined as $\exists t \exists t' (((\bigwedge_{i=1}^e (p'_i(\mathtt{o},{t},t')=0)))\wedge(\bigwedge_{k=1}^{\ell-h}(t_{k}-\mathtt{o}_{b_{k+h}}=0)))$ and $p(\mathtt{o})=0$ are elementarily equivalent modulo ACO and all $L$-polynomials appearing in $\Phi(\mathtt{o})$ are ordered, as desired.
	
	Assume there are $b>0$ left-parenthesis occurring in the $L$-monomial $m(\mathtt{o})$ as in (\ref{eq:nonordered-monomial}). Choose the first right-parenthesis and consider the $L$-monomial $n_1(\mathtt{o})$ enclosed by the preceeding left-parenthesis, that is,
	\[
	m(\mathtt{o})=(((\mathtt{o}_{b_1}^{\beta_{b_1}}(\dots(n_1(\mathtt{o}))\mathtt{o}_{b_i}^{\beta_{b_i}}) \dots )\mathtt{o}_{b_\ell}^{\beta_{b_\ell}}))
	\]
	for some $L$-monomial $n_1(\mathtt{o})$ in which no parenthesis occur. Consider the $L$-monomial $m'(\mathtt{o},\mathtt{s})$ obtained by substituting the $L$-monomial $(n_1(\mathtt{o}))$ by an extra variable $\mathtt{s}_1$ and define $p'(\mathtt{o},\mathtt{s}_1)$ as the $L$-polynomial obtained by replacing the $m(\mathtt{o})$ in $p(\mathtt{o})$ with $m'(\mathtt{o},\mathtt{s}_1)$. Consider the $L$-formula $\psi(\mathtt{o},\mathtt{s}_1)$ defined as:
	\[
	(p'(\mathtt{o},\mathtt{s}_1)=0)\wedge(\mathtt{s}_1-n_1(\mathtt{o})=0).
	\]
	Observe that the formulas $p(\mathtt{o})=0$ and $\exists s_1 \psi(\mathtt{o},s_1)$ are elementarily equivalent modulo ACO. Moreover, by construction, there are exactly $(b-1)$ left-parenthesis occurring the $L$-monomial $m'(\mathtt{o},\mathtt{s})$ of $p'(\mathtt{o},\mathtt{s})$ and there are no parenthesis appearing in the $L$-monomials $\mathtt{s}_1$ and $n_1(\mathtt{o})$. Then, by inductive assumption on the number $(b-1)$ of left-parenthesis occurring in the $L$-monomial $m'(\mathtt{o},\mathtt{s}_1)$ of $p'(\mathtt{o},\mathtt{s}_1)$, there exists an $L$-formula $\psi'(\mathtt{o},\mathtt{s})$ of the form 
	\[
	(p^{(b)}(\mathtt{o},\mathtt{s})=0)\wedge\bigwedge_{i=1}^b (\mathtt{s}_i-n_{i}(\mathtt{o},\mathtt{s}_1,\dots,\mathtt{s}_{i-1})=0)
	\]
	 for some string of new variables $\mathtt{s}:=(\mathtt{s_1},\dots,\mathtt{s}_b)$, for some $L$-monomials $n_{i}(\mathtt{o},\mathtt{s}_1,\dots,\mathtt{s}_{i-1})$ in which no parenthesis occur and for some $L$-polynomial $p^{(b)}(\mathtt{o},\mathtt{s})$ having $a$ non-ordered $L$-monomials but one of them, say $m^{(b)}(\mathtt{o},\mathtt{s})$, does not have any parenthesis, such that $\exists s \psi'(\mathtt{o},s)$ and $p(\mathtt{o})=0$ are elementarily equivalent modulo ACO. Thus, we may order the variables of the non-ordered $L$-monomials without parenthesis $n_{i}(\mathtt{o},\mathtt{s})$ for $i\in\{1,\dots,b\}$ and $m^{(b)}(\mathtt{o},\mathtt{s})$ of $p^{(b)}(\mathtt{o},\mathtt{s})$ as in the case $b=0$ showing that there exists an $L$-formula $\Phi(\mathtt{o},\mathtt{s},\mathtt{t})$ of the form
	\[
	\Big((p^{(b+1)}(\mathtt{o},\mathtt{s},\mathtt{t})=0)\wedge\bigwedge_{j=1}^\ell q_j(\mathtt{o},\mathtt{s},\mathtt{t})\Big)\wedge \bigwedge_{i=1}^{b} \Big((s_i-n'_{i}(\mathtt{o},\mathtt{s},\mathtt{t})=0)\wedge\bigwedge_{k=1}^{\ell_i}(r_k(\mathtt{o},\mathtt{s},\mathtt{t})=0)\Big)
	\]
	 for some string of new variables $\mathtt{t}:=(\mathtt{t_1},\dots,\mathtt{t}_c)$ such that:
	 \begin{enumerate}[label={(\roman*)}, ref=(\roman*)]
	 	\item $\exists t\exists s \Phi(\mathtt{o},s,t)$ and  $p(\mathtt{o})=0$ are elementarily equivalent modulo ACO,
	 	\item\label{en:ind} every $L$-polynomial occurring in $\Phi(\mathtt{o},\mathtt{s},\mathtt{t})$ is ordered apart from $p^{(b+1)}(\mathtt{o},\mathtt{s},\mathtt{t})$, which has exactly $(a-1)$ non-ordered $L$-mononials.
	 \end{enumerate}
	 Thus, by applying the inductive assumption to the atomic $L$-formula $p^{(b+1)}(\mathtt{o},\mathtt{s},\mathtt{t})=0$, we get an $L$-formula of the form $\exists s \exists t \exists t' (\bigwedge_{i=1}^e (p_i(\mathtt{o},s,t,t')=0))$ for some string of new variables $\mathtt{t}':=(\mathtt{t}'_1,\dots,\mathtt{t}'_d)$ and for some ordered $L$-polynomials $p_i(\mathtt{o},\mathtt{s},\mathtt{t},\mathtt{t}')$ for $i\in\{1,\dots,e\}$, which is elementarily equivalent to $p(\mathtt{o})=0$ modulo ACO, as desired.
\end{proof}

\begin{remark}\label{rmk:ordered-def}
	A similar statement analogous to Theorem \ref{thm:ord-form} holds for $L'$-formulas, where $L'$ denotes $L\cup\{q_i\}_{i\in\hh_R}$ and $L\cup\{o_i\}_{i\in\oo_R}$, respectively. Indeed, it suffices to apply the same substitution procedure to the coefficients that may occur in non-ordered $L'$-monomials between two variables.
\end{remark}

As a direct consequence of Theorem \ref{thm:ord-form} we deduce the following result.

\begin{corollary}
	The $L$-theories $\textnormal{ACQ}$ and $\textnormal{ACO}$ do not have quantifier elimination for the fragment of ordered $L$-formulas.
\end{corollary}

Latter corollary shows that restricting our interest to ordered $L$-formulas is not useful to understand which $L$-formulas admit quantifier elimination. On the contrary, when restricting to quantifier-free $L$-formulas with parameters in a fixed model, then there are sets that are definable by means of a quantifier-free $L$-formula with parameters but they are not definable by any ordered quantifier-free $L$-formula with parameters, as it is shown by the following example.

\begin{example}
	Let $R\models \textnormal{RCF}$ and $L_{\hh_{R}}:=L\cup\{c\}_{c\in\hh_R}$. Consider the subset $\mathcal{Z}_{\hh_R}(p)\subset\HH_R$, with $p:=\mathtt{q}-i\mathtt{q}i-j\mathtt{q}j-k\mathtt{q}k$, as in Remark \ref{rmk:zero-set}. By Theorem \ref{thm:FTA}, the zero set of a non-null ordered $L_{\hh_{R}}$-polynomial in one variable has real dimension $0$ or $2$ as real algebraic subsets of $R^4$. As a consequence, those subsets of $\HH_R$ defined by quantifier-free $L_{\hh_{R}}$-formulas have real dimension $d\in\{0,1,2,4\}$ as locally closed subsets of $R^4$. Since $\dim_R(X)=3$, we deduce that $X$ is quantifier-free $L_{\hh_{R}}$-definable but it is not ordered quantifier-free $L_{\hh_{R}}$-definable. 
 \end{example}
 
 Observe that the proof of Theorem \ref{thm:ord-form} in the case of atomic $L$-formulas, together with Remark \ref{rmk:ordered-def}, has a direct geometric interpretation since it allows to characterize (basic) algebraic sets, both over $\hh_R$ and $\oo_R$, up to real biregular isomorphism.

 \begin{theorem}\label{thm:iso-real}
 	Every real algebraic subset $X\subset R^n$ can be realized by a basic algebraic subset $X_\hh$ of $\hh_{R}^m$ and $X_\oo$ of $\oo_{R}^{m'}$, for some $m\geq 4n$ and $m'\geq 8m$. More precisely, there exist basic algebraic sets $X_\hh\subset\hh_{R}^m$ and $X_\oo\subset\oo_{R}^{m'}$ such that $X\subset R^n$, $\jmath^m(X_\hh)\subset R^{4m}$ and $\jmath^{m'}(X_\oo)\subset R^{8m'}$ are biregularly isomorphic as real algebraic sets.
 \end{theorem}
 
 \begin{proof}
 	We just prove the quaternionic case, a similar argument works for octonions as well just by using the formulas (\ref{eq:bi-o}) instead of (\ref{eq:bi-q_1})\&(\ref{eq:bi-q_2}) in what follows. Let $X\subset R^n$ be a real algebraic set. Consider $s\in\N$ be the smallest integer such that $n\leq 4s$ and consider $X':=X\times\{0\}\subset R^n\times\{0\}\subset R^n\times R^{4s-n}=R^{4s}$. Of course $X$ and $X'$ are biregularly isomorphic. Choose $p\in R[\mathtt{x}_1,\dots,\mathtt{x}_{4s}]$ such that $\mathcal{Z}_{R^{4s}}(p)=X'$, for instance by taking the sum of squares of a set of generators of its vanishing ideal. Let $\jmath:\hh_R\to R^4$ be the interpretation of Proposition \ref{prop:bi-int} and consider the set $(j^s)^{-1}(X')\subset\hh_R^s$. Clearly, by Corollary \ref{cor:bi-int}, $(j^s)^{-1}(X')\subset\hh_R^s$ is $L_{\hh_R}$-definable. On the other hand, since $X'\subset R^{4s}$ is algebraic, we have an explicit simple description of $(j^s)^{-1}(X')\subset\hh_R^s$. Define the (non-ordered) $L_{\hh_R}$-polynomial $p'(\mathtt{q}_1, \dots, \mathtt{q}_s)$ just by substituting the real variables of $p(\mathtt{x}_1,\dots,\mathtt{x}_{4s})$ in the following way:
 	\begin{align*}
		\mathtt{x}_{4t+1}&=\frac{1}{4}(\mathtt{q}_t-i\mathtt{q}_t i-j\mathtt{q}_t j-k\mathtt{q}_t k),\quad \mathtt{x}_{4t+2}=\frac{1}{4i}(\mathtt{q}_t-i\mathtt{q}_t i+j\mathtt{q}_t j+k\mathtt{q}_t k),\\
 		\mathtt{x}_{4t+3}&=\frac{1}{4j}(\mathtt{q}_t+i\mathtt{q}_t i-j\mathtt{q}_t j+k\mathtt{q}_t k,),\quad \mathtt{x}_{4t+4}=\frac{1}{4k}(\mathtt{q}_t+i\mathtt{q}_t i+j\mathtt{q}_t j-k\mathtt{q}_t k)
 	\end{align*}
 	for every $t\in\{0,\dots,s-1\}$. Hence, $(j^s)^{-1}(X')=\mathcal{Z}_{\hh_{R}^s}(p')$. Then, an application of the proof of Theorem \ref{thm:ord-form}, together with Remark \ref{rmk:ordered-def}, gives a basic algebraic set $X_{\hh}\subset\hh_R^m$ such that $\pi_\hh(X_\hh)=(j^s)^{-1}(X')$, where $\pi_\hh:\hh_R^m\to\hh_R^s$ denotes the projection onto the first $s$ coordinates, since existential quantifiers correspond to projections. Observe that, since the strategy of Theorem \ref{thm:ord-form}'s proof relies in substituting new variables to non-ordered $L$-monomials appearing in $p'(\mathtt{q})$, the resulting basic algebraic set $X_\hh\subset \hh_R^m$ is actually the graph of a ordered $L_{\hh_R}$-polynomial function $P:\hh_R^s\to\hh_R^m$ restricted to $(j^s)^{-1}(X')$. Thus, $\pi_\hh|_{X_\hh}$ is invertible and its inverse $(\pi_\hh|_{X_\hh})^{-1}:X'\to X_\hh$ has entries given by ordered $L_{\hh_R}$-polynomials. This ensures that $\pi|:\jmath^m(X_\hh)\to X$ is a biregular isomorphism between real algebraic sets, where $\pi:R^{4m}\to R^{4s}$ is the canonical projection onto the first $4s$ coordinates.
 \end{proof}
 
\section*{Acknowledgements}

The author would like to thank Stefano Baratella, Antonio Carbone and Riccardo Ghiloni for valuable discussions and remarks during the drafting process. The author would like to thank the anonymous referee for his careful reading of the manuscript and for his suggestions, which have improved the presentation of this article.

\printbibliography

\end{document}